\documentclass{article}
\usepackage{amsmath,amssymb,verbatim}
\usepackage{times}
\usepackage{color}

\usepackage[dvips]{graphicx}
\usepackage{amsfonts}
\usepackage{amsthm}
\usepackage{epsfig}
\usepackage{color}

\theoremstyle{plain}
\newtheorem{theorem}{Theorem}[section]

\newtheorem{lemma}[theorem]{Lemma}
\newtheorem{proposition}[theorem]{Proposition}
\newtheorem{corollary}[theorem]{Corollary}
\newtheorem{algorithm}[theorem]{Algorithm}

\theoremstyle{definition}
\newtheorem{definition}[theorem]{Definition}
\newtheorem{remark}[theorem]{Remark}

\newtheorem{question}[theorem]{Question}

\newtheorem{example}[theorem]{Example}
\newtheorem{conjecture}[theorem]{Conjecture}

\newtheorem*{notation}{Notation}
\newtheorem{counterexample}[theorem]{Counterexample}

\newcommand{\todo}[1]{\vspace{5 mm}\par \noindent
\marginpar{\textsc{ToDo}}
\framebox{\begin{minipage}[c]{0.95 \textwidth}
\tt #1 \end{minipage}}\vspace{5 mm}\par}

\renewcommand{\todo}[1]{}

\newcommand{\idiot}[1]{\vspace{5 mm}\par \noindent
\framebox{\begin{minipage}[c]{0.95 \textwidth}
\tt #1 \end{minipage}}\vspace{5 mm}\par}

\renewcommand{\idiot}[1]{}

\newcommand{\xs}{x_1,\ldots,x_n}                

 \newcommand{\I}{{\mathcal{I}}}                  

\newcommand{\erase}[1]{}
\newcommand{\lm}{\lambda}
\newcommand{\Il}{{\mathcal{I}}_{\lambda}}
\newcommand{\Jl}{{\mathcal{J}}_{\lambda}}
\newcommand{\Ml}{{\mathcal{M}}_{\lambda}}
\newcommand{\El}{{\mathcal{E}}_{\lambda}}
\newcommand{\Kl}{{\mathcal{K}}_{\lambda}}
\newcommand{\J}{{\mathcal{J}}}
\newcommand{\E}{{\mathcal{E}}}

\newcommand{\mystack}[2]{\stackrel{\scriptscriptstyle{#1}}{\scriptscriptstyle{#2}}}
\newcommand{\mysumstack}[2]{\stackrel{\scriptstyle{#1}}{\scriptstyle{#2}}}

\definecolor{shade}{gray}{0.75}
\newcommand{\collbox}[1]{\framebox{$#1$}}
\newcommand{\collshadebox}[1]{\let\fboxsave\fboxsep\def\fboxsep{0mm}{\framebox{\let\fboxsep\fboxsave\colorbox{shade}{$#1$}}}}
\newcommand{\shadebox}[1]{\colorbox{shade}{$#1$}}

\def\Young#1{\vbox{\smallskip\offinterlineskip
    \halign{&\vbox{##}\kern-\Thickness\cr #1}}}

\newdimen\Squaresize \Squaresize=12pt
\newdimen\Thickness \Thickness=.3pt
\newdimen\Correction \Correction=7pt

\def\Vide#1{\hbox{
       \vbox to \Squaresize{\vss
          \hbox to \Squaresize{\hss#1 \hss}\vss}
    \hskip-\Correction}
   \kern-\Thickness}

\def\Carre#1{\hbox{\vrule width \Thickness
   \vbox to \Squaresize{\hrule height \Thickness\vss
      \hbox to \Squaresize{\hss#1\hss}
   \vss\hrule height\Thickness}
   \unskip\vrule width \Thickness}
   \kern-\Thickness}

\def\Box#1{\Carre{$\scriptstyle#1$}}

 \date{\today}

\author{Riccardo Biagioli\thanks{Institut Camille Jordan, UMR 5208 du CNRS, Universit\'e
de Lyon, Universit\'e Lyon 1,  biagioli@math.univ-lyon1.fr} \and Sara
Faridi\thanks{Department of Mathematics, Dalhousie University,
Halifax, Canada, faridi@mathstat.dal.ca (research supported by NSERC)}
\and Mercedes Rosas\thanks{Departamento de \'Algebra, Universidad de
Sevilla, mrosas@us.es (research supported by a Ram\'on y Cajal
grant, MEC)}}

\title{The defining ideals of conjugacy classes of nilpotent matrices and
a conjecture of Weyman}

\begin{document}

\maketitle

\begin{abstract} Tanisaki introduced generating sets  
 for the defining ideals of the schematic intersections of the closure
 of conjugacy classes of nilpotent matrices with the set of diagonal
 matrices. These ideals are naturally labeled by integer partitions.
 Given such a partition $\lambda$, we define several methods to
 produce a reduced generating set for the associated ideal $\Il$.  For
 particular shapes we find nice generating sets.  By
 comparing our sets with some generating sets of $\Il$ arising from a
 work of Weyman, we find a counterexample to a related conjecture of
 Weyman.
\end{abstract}

\section{Introduction}

\idiot{Convention : the partition $\lambda'$ gives the size of the
blocks, and that the set of matrices with Jordan canonical form
$\lambda$ is denoted by $O(\lambda')$, and its ideal by
$\I_{\lambda'}$.  Weyman follows the opposite convention for the ideals
but the same for the varieties.  Haiman follows our convention}

Let $X$ be the set of $n\times n$ matrices over a field $k$ of
characteristic $0$.  In his paper Kostant~\cite{K} showed that the
ideal of polynomial functions vanishing on the set of nilpotent
matrices in $X$, is given by the invariants of the action by
conjugation of $GL(n)$ on $X$.  Let $C_{\lambda}$ be the conjugacy
class of nilpotent matrices in $X$ having Jordan block sizes
$\lambda'_1, \ldots , \lambda'_h$, with $\lambda$ a
partition of $n$ and $\lambda'$ its transpose.  Let
$\overline{C}_{\lambda}$ be the nilpotent orbit variety defined as the
Zariski closure of $C_{\lambda}$. De Concini and Procesi~\cite{DP}
asked for a description of the ideal $\Jl$ of polynomial functions
vanishing on $C_{\lambda}$, for a general partition $\lambda$. They
were interested in a refinement of Kostant's result, which corresponds
to the case $\lambda=(1^n)$.  De Concini and Procesi described a set
of elements of $\Jl$ that they conjectured to be a generating set.
Later, Tanisaki~\cite{T} conjectured a simpler generating set, and
Eisenbud and Saltman \cite{ES} generalized Tanisaki's conjecture to
rank varieties. Finally, in 1989 Weyman \cite{W1} used geometric
methods to show that the three conjectures hold, and conjectured a
minimal generating set $\mathcal{W}_\lambda$ for these ideals.

 In the present paper we focus on a related family of ideals that we
 denote by $\I_{\lambda}$ and call \emph{De Concini-Procesi ideals}.  These
 are the ideals of the scheme-theoretic intersection of nilpotent
 orbit varieties $\overline{C}_{\lambda}$ with the set of diagonal
 matrices.  De Concini and Procesi \cite{DP} produced a set of
 generators for these ideals that was later simplified by
 Tanisaki~\cite{T}. In both cases, the sets of generators are highly
 nonminimal.  In the case $\lambda=(1^n)$, Kostant's theorem implies
 that the elementary symmetric functions of the eigenvalues of the
 matrices give a minimal set of generators for $\I_{(1^n)}$.

Our work in this paper is motivated by the search for a minimal
generating set for De Concini--Procesi ideals. To this end, we
simplify the generating set described by Tanisaki using elementary
facts of the theory of symmetric functions.  We provide several
reduction methods. The obtained sets are minimal in special cases, and
are generally much smaller.  The main tool we use is a special filling
of the Young diagram of the partition $\lambda$ which we call the
\emph{regular filling}.

Clearly, by adding the defining ideal of the diagonal matrices to any
generating set for the ideal $\J_\lambda$, we obtain a generating set
for $\Il$. The following question is natural: Is it true that, after
adding these generators to Weyman's conjectured minimal generating set
for $\Jl$, a minimal generating set for $\Il$ is obtained ? We give a
negative answer to this question and provide some infinite families
of counterexamples.  With the help of Macaulay 2 we verify that one of
these counterexamples is also a counterexample to the original
conjecture of Weyman on a minimal generating set of $\Jl$. This has
been a well studied problem that has been open for the past seventeen
years. We hope that our methods together with those of Weyman will
eventually lead to a complete solution of the problem of finding a
minimal generating set for both ideals $\Il$ and $\Jl$.

 Our paper is organized as follows. In Section~\ref{s:basictools} we
 introduce some basic tools from the theory of symmetric functions.
 In Section~\ref{s:dp-ideals}, we introduced Tanisaki's generating set
 for the De Concini-Procesi ideal, and derive a simple combinatorial
 description for it. This leads to a simple rule to read a set of
 generators of the ideal directly from a special filling of the Young
 diagram of the partition that call the \emph{regular filling}.  In
 Section~\ref{s:top-cells} we show that only generators read from the
 top entries of the regular filling are necessary in order to
 construct a generating set for $\Il$. The resulting generating set is
 in a one-to-one correspondence with a generating set that arises from
 the work of Weyman~\cite{W1}.  In the case where the partition
 $\lambda$ is a hook, our result coincides with the minimal generating
 set we introduced in \cite{BFR}. For a general shape though, this
 generating set could be far from minimal.  In
 Section~\ref{s:towardsU} we reduce the number of generators coming
 from each column of the Young diagram. Finally in
 Section~\ref{ultima}, we provide many examples and counterexamples to
 the modified version of Weyman's conjecture, and discuss classes
 where our reductions work best. Inside those families we are able to
 find a counterexample to the original conjecture of Weyman on a
 minimal generating set for the ideal $\J_\lambda$. Throughout the
 paper, we raise new questions whose answers could help illuminate the
 problem of finding minimal generating sets for $\Il$ and $\Jl$.

\section{Basic Tools} \label{s:basictools}

We will be working in the polynomial ring $R=k[\xs]$, where $k$ may
be an arbitrary field of characteristic $0$.

We define a {\em partition} of $n \in \mathbb{N}$ to be a finite
sequence $\lambda=(\lambda_1,\ldots,\lambda_k) \in \mathbb{N}^k$, such
that $\sum_{i=1}^k \lambda_i =n$ and $\lambda_1 \geq \ldots \geq
\lambda_k$. If $\lambda$ is a partition of $n$ we write $\lambda
\vdash n$. The nonzero terms $\lambda_i$ are called {\em parts} of
$\lambda$. The number of parts of $\lambda$ is called the {\em length}
of $\lambda$, denoted by $\ell(\lambda)$, so $\lm_i=0$ if $i>\ell(\lm)$.

Let $\lambda=(\lambda_1,\ldots,\lambda_k)$ be a partition of $n$.  The
{\em Young diagram} of a partition $\lambda$ is the left-justified array with
$\lambda_i$ squares in the $i$-th row, from bottom to top. We use the symbol
$\lambda$ for both a partition and its associated Young diagram. For
example, the diagram of $\lambda=(4,4,2,1)$ is illustrated in
Figure~\ref{Ytableau} on the left.

For a partition $\lambda=(\lambda_1,\ldots,\lambda_k)$ we define its
{\em conjugate} partition as $\lambda^\prime=(\lambda_1^\prime,\ldots,
\lambda_h^\prime)$, where for each $i\geq 1$, $\lambda_i^\prime$ is
the number of parts of $\lambda$ that are bigger than or equal to $i$.
The diagram of $\lambda^\prime$ is obtained by flipping the diagram of
$\lambda $ across the diagonal.
{\newdimen\Squaresize \Squaresize=10pt
\newdimen\Thickness \Thickness=.3pt
\newdimen\Correction \Correction=7pt
\begin{figure}[h]
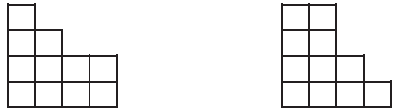

\begin{eqnarray*}
&\begin{matrix} \Young{ \Box{}\cr \Box{}& \Box{} \cr \Box{}& \Box{}&
               \Box{}& \Box{}\cr \Box{}& \Box{}& \Box{}& \Box{}&
                \cr }
     \end{matrix} \;\;\;\;\;\;\;\;\;\;\;\;\;\;\;\;\;\;\;
 &\begin{matrix} 
 \Young{ \Box{} & \Box{}\cr 
                \Box{}& \Box{} \cr 
                \Box{}& \Box{}&
               \Box{}&\cr \Box{}& \Box{}& \Box{}& \Box{}&
                \cr }
     \end{matrix}
 \end{eqnarray*}
\caption{The partition $\lambda=(4,4,2,1)$ and its conjugate 
     $\lambda^\prime=(4,3,2,2)$.}\label{Ytableau}
     
\end{figure}

We shall need some basic definitions from the theory of symmetric
functions. First, we introduce the generating series for the
elementary and the complete symmetric polynomials (denoted
respectively by $E(S,z)$ and $H(S,z)$).  These series are defined as:
\begin{align}\label{e:generating-functions}
&&E(S,z)= \sum_{i \ge 0} z^i e_i(S)= \prod_{a \in S} (1+za),  &&\text{ and }
 && H(S,z)= \sum_{i \ge 0} z^i h_i(S)= \prod_{a \in S} \frac{1}{1-za} ,
\end{align}
where $S$ is a set of variables, and $z$ is a formal variable.
Therefore, the {\em elementary symmetric polynomial} $e_r(S)$ is the
sum of all square free monomials of degree $r$ in the variables of
$S$, and the {\em complete symmetric polynomial} $h_r(S)$ is the sum
of all monomials of degree $r$ in the variables of $S$.

In order to introduce the monomial symmetric polynomials
$m_{\lambda}(S)$, we say that a monomial ${\bf x}^s=x_1^{s_1} x_2^{s_2}
\cdots x_n^{s_n}$ has {\em type} $\lambda$, if the partition $\lambda$
is obtained by rearranging the sequence $(s_1, s_2, \ldots, s_n)$ in
weakly descending order.  Given a partition $\lambda$, the {\em
monomial symmetric polynomial} $m_{\lambda}=m_\lambda(S)$ is defined
as
\[ m_{\lambda}(S) = \sum \mathbf{x}^{s} \]
where the sum is taken over all different monomials $\mathbf{x}^{s}$ of type $\lambda$ and with all variables in $S$.

If $f \in k[\xs]$ is a symmetric polynomial, and $S \subseteq
\{x_1,\ldots,x_n\}$, we define $f(S)$ as the evaluation of $f$ at the
set $S$, by setting all variables $x \in \{x_1,\ldots,x_n\} \setminus
S$ to be equal to $0$ in $f$. For instance, $e_2(x_1, x_3)=x_1x_3$.
The polynomial $f(S)$ is called a {\em partially symmetric
  polynomial}. In general, it is no longer invariant under the action
of the symmetric group on $n$ letters.

For simplicity, given a symmetric polynomial $f \in k[\xs]$, for all
$1 \leq k \leq n$, we will denote by $f(k)$ the following set of
partially symmetric polynomials,
\[
f(k)=\{ f(S) \mid S \subseteq \{\xs\}, \ |S|=k \}.
\]
For example, let $n=4$, then
$e_2(3)=\{x_1x_2 + x_1x_3 + x_2x_3, \ x_1x_2 + x_1x_4 + x_2x_4, \
x_1x_3 + x_1x_4 + x_3x_4, \ x_2x_3 + x_2x_4 + x_3x_4\}.$
Note that if $r>k$ we have $e_r(k)=\emptyset$. 
 
\begin{notation}
Let $S \subseteq \{\xs\}$. For $x \in S$, and
$I=\{x_{i_1},\ldots,x_{i_k}\}\subseteq S$, we let
\begin{align*} 
S_x &= S \setminus \{x\} \ \ {\rm and} \ \  S_{i_1,\ldots,i_k}= S \setminus I.
\end{align*}
\end{notation}

We shall be using the following elementary lemma later in the paper.

\begin{lemma}[Basic Lemma]\label{l:basic-lemma} Let $S \subseteq 
\{\xs\}$, $|S|=s$, and let $j\leq s$.  Then
\begin{enumerate}
\item  $e_j(S)=e_j(S_x)+xe_{j-1}(S_x)$ for all
 $x \in S$;

\item  $\displaystyle \sum_{x\in S}e_j(S_x)=(s-j)e_j(S)$;

\item  $\displaystyle \sum_{x\in S}xe_{j-1}(S_x)=je_j(S)$.

\end{enumerate}
\end{lemma}

\begin{proof}  
\begin{enumerate}
\item Clear.

\item Fix a square-free monomial $M$ of degree $j$ appearing in
$e_j(S)$. Without loss of generality, assume $M=x_1 \cdots x_j$ and
$S=\{x_1,\ldots,x_s\}$.  Then each $e_j(S_{x_t})$ contains exactly one
copy of $M$, for $t=j+1,\ldots,s$. There are exactly $s-j$ such
indices $t$, so $M$ appears $s-j$ times in the left-hand sum.

\item  We use the equation in Part 1, and sum over all elements of $S$ :
$\sum_{x\in S}e_j(S)=\sum_{x\in S}e_j(S_x)+\sum_{x\in
S}xe_{j-1}(S_x)$ so by Part 2 we have 
$se_j(S)=(s-j)e_j(S)+\sum_{x\in S}xe_{j-1}(S_x)$ and hence 
$je_j(S)=\sum_{x\in S}xe_{j-1}(S_x).$

\end{enumerate}
\end{proof}

\begin{proposition}[Another presentation of the partially symmetric polynomials]\label{p:new-presentation} 
Let $S=\{\xs\}$, $i\leq n$, and define the ideal
$\E_i(S)=(e_1(S),\ldots,e_i(S))$ in the polynomial ring $k[\xs]$.  Let
$U \subseteq S$ be a subset of cardinality $u$. Then for $i \leq n-u$
we have
\begin{equation}\label{RHS22}
e_i(S \setminus U)= (-1)^i h_i(U)  \mbox{ mod } \E_i(S).
\end{equation}
\end{proposition}

\begin{proof} This result follows from a formal manipulation of the generating functions in (\ref{e:generating-functions}). We have 
\[
E(S \setminus U,z) = \prod_{\substack{a \in S \\ a \not \in U}} (1+za)
= \frac{\prod_{a \in S } (1+za) }{\prod_{a \in U } (1+za)} = E(S,z)
H(U,-z).
\]
Therefore, extracting the coefficient of $z^i$ from both sides of  the resulting equation 
$E(S \setminus U,z) =E(S,z) H(U,-z)$
we obtain 
\[
e_i(S \setminus U) = \sum_{j=0}^i e_j(S) (-1)^{i-j} h_{i-j}(U).
\]
By hypothesis $e_j(S)$ is in the ideal for $j=1, \ldots, i$.  Since
$e_0(S)=1$, the result follows.
\end{proof}

\section{A new combinatorial description of Tanisaki's generating set for $\Il$}
\label{s:dp-ideals}

In this section, we define a family of ideals $\Il$ in the polynomial
ring $R=k[\xs]$ indexed by partitions $\lambda$ of $n$. The ideal
$\Il$ was first introduced by De Concini and Procesi~\cite{DP} in
order to describe the coordinate ring of the schematic intersection of
the Zariski closure of the conjugacy class of nilpotent matrices of
shape $\lambda$, with the set of diagonal matrices.

In order to manipulate De Concini-Procesi ideals, we use a generating
set defined by Tanisaki~\cite{T}. A nice feature of Tanisaki's
generating set is that its elements are elementary partially symmetric
polynomials. Furthermore, Tanisaki's proof of the correctness of his
generating set is both elegant and elementary, and it is based on
standard linear algebra facts. Finally, Tanisaki's generating set has
proven to be very fruitful in algebraic combinatorics, see for example
\cite{AB,BG,GP}.

Let $\lambda=(\lambda_1,\ldots,\lambda_k)$ be a partition of $n$.  For
the purpose of the next formula, we add enough zeroes to the end of
$\lm$ so that it has $n$ terms:
$\lm=(\lambda_1,\ldots,\lambda_n)$. For any $1\leq k \leq n$, we
define
\begin{equation}\label{def:delta}
\delta_k(\lambda)=\lambda'_n + \lambda'_{n-1} + \ldots + 
  \lambda'_{n-k+1}.
 \end{equation}
It is clear that $\delta_n(\lambda)\geq
\delta_{n-1}(\lambda)\geq\ldots \geq \delta_1(\lambda)$, and that
$\delta_n(\lambda)=n$.

\begin{theorem}[Tanisaki's generating set~\cite{T}]\label{t:DP} 
The ideal $\mathcal{I}_\lambda$ is generated by the following
collection of elementary partially symmetric polynomials
\begin{equation}\label{generators}
\mathcal{I}_\lambda=\big( e_r(k) \mid k=1, \ldots, n, \; {\rm and} \ \
k \geq r > k -\delta_k(\lambda)\big).
\end{equation}
\end{theorem}

\begin{definition}[De Concini-Procesi ideal] We call the ideal 
$\mathcal{I}_\lambda$ defined in Theorem~\ref{t:DP} the {\em De
Concini-Procesi ideal} of the partition $\lm$.

\end{definition}
 
Since for any partition $\lambda$ of $n$, $\delta_n(\lambda)=n$, when
we set $k=n$ in (\ref{generators}) we conclude that
$\mathcal{I}_\lambda$ contains all the elementary
symmetric polynomials in all the variables $x_1,\ldots,x_n$.

\begin{example}\label{ex1} Let $\lambda=(4,4,2,1,0,0,0,0,0,0,0)\vdash 11$  
be the partition appearing in Figure~\ref{Ytableau}. Then
$(\delta_1(\lambda),\ldots,\delta_{11}(\lambda))=(0,0,0,0,0,0,0,2,4,7,11)$. Hence
\[(1-\delta_1(\lambda),\ldots,11-\delta_{11}(\lambda))
   =(1,2,3,4,5,6,7,6,5,3,0).\] Here $n=11$. For $k=1,\ldots,7$ there is no
admissible $e_r(k)$ in the generating set described
in (\ref{generators}).  So the generating set of $\mathcal{I}_{(4421)}$ consists of the
following elements
\begin{eqnarray*}
\begin{tabular}{l|l|l}
{\bf }& {\bf Generators} \\
\hline&\\ 
$k=8$ & $e_7(8),e_8(8)$ \\
$k=9$ & $e_6(9), e_7(9), e_8(9), e_9(9) $ \\
$k=10$ & $e_4(10),e_5(10),\ldots, e_{10}(10)$ \\
$k=11$ & $e_{1}(11),e_2(11),\ldots, e_{11}(11)$\\
\end{tabular}
\end{eqnarray*}
\end{example}

We now give a simple combinatorial description of the set of
generators for $\Il$ described in Theorem~\ref{t:DP}, and then
demonstrate how to shorten it so that one can read a reduced
generating set for $\Il$ directly from the diagram of the partition
$\lm$. In order to do so we introduce the notion of regular filling.

\begin{definition}[The regular filling of a partition]\label{d:regular-filling} 
      Let $\lambda$ be a partition of $n$. Draw its Young diagram and
      then fill its cells with the numbers $1,2,\ldots,n$ from top to
      bottom and from left to right, skipping the cells in the bottom
      row, which should be filled at the end from right to left. This
      is called the {\em regular filling} of $\lambda$, denoted {\it
        rf}.
\end{definition}

\begin{figure}
\begin{eqnarray*}
&
  \begin{matrix} 
	           \Young{     
	            \Box{1} &    \cr
	             \Box{2} & \Box{4}    \cr
	             \Box{3} & \Box{5} & \Box{6} &\Box{7}     \cr
	              \Box{11}  & \Box{10}  & \Box{9}  & \Box{8} \cr 
	                }
  \end{matrix}
\end{eqnarray*}
\caption{The regular filling of $(4,4,2,1)$.}\label{f:regular-filling}	
\end{figure}

\begin{definition}[The reading process]\label{d:reading-process} 
	We associate to any filling $f$ of the Young diagram of
        $\lambda$ a set of partial symmetric polynomials, denoted by
        $\mathcal{G}_f(\lm)$. We read the elements of this set from
        the filling as follows. For a given column of $\lambda$ we add
        to $\mathcal{G}_f(\lm)$ all the elements of the sets $e_r(k)$,
        where $k$ is the entry in the bottom cell of the column, and
        the degrees $r$'s are given by all the entries in that column.
\end{definition}

\begin{notation}
 From now on, we enumerate columns and rows of a Young diagram from left to right by starting from zero. So the ``first'' column will be the $0$-th column; similarly for rows.
\end{notation}

\begin{example}
	For the partition $\lambda=(4,4,2,1)$, the regular filling
	{\em rf} is illustrated in Figure~\ref{f:regular-filling}.
	The reading process of this filling gives the set
	$\mathcal{G}_{\it rf}(\lm)$ consisting of: the elementary
	symmetric polynomials $e_1(x_1, \ldots, x_{11})$, $e_2(x_1,
	\ldots, x_{11})$, $e_3(x_1, \ldots, x_{11})$, $e_{11}(x_1,
	\ldots, x_{11})$, coming from the $0$-th column; the partially
	symmetric polynomials of the sets $e_4(10), e_5(10),
	e_{10}(10)$ read from the first column, $e_6(9), e_9(9)$ from
	the second column, and $e_7(8), e_8(8)$ from the last column.
\end{example}

	By using this reading process, we are going to read Tanisaki's
	generators from a special filling.

\begin{definition}[The antidiagonal filling]\label{d:antidiagonal-filling}
        Let $\lambda$ be a partition of $n$. Compute the partition
	$\delta(\lambda)$
	$$ \delta(\lambda)=\delta_n(\lambda)\geq
	\delta_{n-1}(\lambda)\geq \ldots \geq \delta_1(\lambda),$$ 
	where $\delta_k(\lambda)$ is defined as in (\ref{def:delta}), 
	and draw the Young diagram of its conjugate $\delta'(\lm)$. Now
	fill the $0$-th column of $\delta'(\lm)$ by $1,2,\ldots,n$ from
	top to bottom, and then fill the remainder of the diagram so
	that the filling is constant following each antidiagonal. We
	call this the {\em antidiagonal filling} of $\delta'(\lm)$ and 
	denote it by {\it af}. 
\end{definition}

        For our running example $\lm=(4,4,2,1,0^7)$, we have
	$\delta(\lm)=(11,7,4,2,0^7)$; the antidiagonal filling of
	$\delta'(\lm)$ is given in
	Figure~\ref{f:antidiagonal-filling}. Note that the bottom
	entry of the $k$-th column of $\delta'(\lm)$ is $n-k$.

\begin{figure}[h]
 \begin{align*}
	    \begin{matrix} 
	                 \Young{     
	                 \Box{1}      \cr
	                 \Box{2}      \cr
	                 \Box{3}      \cr
	                 \Box{4}      \cr
	                 \Box{5} & \Box{4}      \cr
	                 \Box{6} & \Box{5}      \cr
	                 \Box{7} & \Box{6}      \cr
	                 \Box{8} & \Box{7} & \Box{6}  \cr
	                 \Box{9} & \Box{8} & \Box{7}     \cr
	                 \Box{10} & \Box{9} & \Box{8} & \Box{7}     \cr
	                 \Box{11}  & \Box{10}  & \Box{9} & \Box{8}  \cr}
	      \end{matrix} 
\end{align*}
\caption{The antidiagonal filling of
$\delta'(\lm)$.}\label{f:antidiagonal-filling}
\end{figure}

Let $\lambda$ be a partition of $n$. Compute the set $\mathcal{G}_{\it
af}(\delta'(\lambda))$ by applying the reading process to the
antidiagonal filling {\it af} of $\delta'(\lambda)$. We have the
following lemma.
\begin{lemma} Let $\lambda$ be a partition of $n$. Then Tanisaki's 
set of generators is $\mathcal{G}_{\it af}(\delta'(\lambda))$.  In
particular,
$$\Il=(\mathcal{G}_{\it af}(\delta'(\lambda))).$$
\end{lemma}

\begin{proof} Let $\lm=(\lm_1,\ldots,\lm_{n})$. Compute $\delta'(\lambda)$
 and fill its diagram with the antidiagonal filling. According to
 Theorem~\ref{t:DP}, to compute Tanisaki's generating set, we need to
 find for which $k$ the interval $[k-\delta_k(\lm)+1,\ldots, k-1, k]$
 is nonempty; clearly this happens when $\delta_k(\lm)>0$.

From the definition of $\delta_k(\lambda)$, the only times
$\delta_k(\lm)>0$ is when $k=n-\lm_1+1, \ldots, n$. So we are
considering values $e_r(S)$ for sets $S$ such that $n-\lm_1+1 \leq |S|
\leq n$. This is an interval of length $\lm_1$, and the numbers
$k=|S|$ we are considering are exactly the entries in the first row of
$\delta'(\lm)$.

Now, fix a column $t$ that has entry $n-t$ in its bottom cell.  The
generating set described in Theorem~\ref{t:DP} has $e_r(S)$, where
$|S|=n-t$ and $r=n-t - \delta_{n-t}(\lm)+1, \ldots, n-t$.
Note that there exactly $\delta_{n-t}(\lm)$ values that $r$ takes,
and that is exactly the size of the $t$-th column of
$\delta'(\lm)$. The mentioned values of $r$ are exactly the entries of
the $t$-th column of the antidiagonal filling of $\delta'(\lm)$.
\end{proof}

One can easily check that this procedure applied to the antidiagonal
filling in Figure~\ref{f:antidiagonal-filling} produces the generators
given in the table of Example~\ref{ex1}.

We are now able to show the main result of this section, namely, that
$\Il$ is the sum of three simpler ideals. In order to do so we will
use the regular filling.

\begin{theorem}\label{th:regular} Let $\lm$ be a partition of $n$. 
Fill the diagram of $\lm$ with the regular filling, and compute the
set $\mathcal{G}_{\it rf}(\lm)$ by using the reading process described
in Definition~\ref{d:reading-process}. Then
$$\mathcal{I}_\lm=(\mathcal{G}_{\it rf}(\lm)).$$
\end{theorem}

\begin{proof} Compute the partition $\delta'(\lm)$, fill its diagram 
with the antidiagonal filling and read off all of Tanisaki's generators.
By Part 2 of Lemma~\ref{l:basic-lemma}, if $e_r(x_1,\ldots, x_j)\ne 0$
belongs to the ideal, so does $e_r(x_1, \ldots, x_J)$ for any $J > j$.
Therefore, for each entry $r =1,\ldots, n$, we only need to keep the
generators coming from the rightmost occurrence of that $r$ in the
antidiagonal filling of $\delta'(\lm)$. So we delete all other
occurrences of $r$ in that filling, and the corresponding cell. We
obtain a filling that contains exactly one occurrence of each of the
numbers from $1$ to $n$.  Now observe that the differences of heights
between adjacent columns of $\delta'(\lm)$ are given by the sequence
$\lm^\prime_1,\ldots,\lm^\prime_{\lm_1}$. So after the deletion
process, explained above, the remaining diagram will have columns of
height $\lm^\prime_1,\ldots,\lm^\prime_{\lm_1}$. Hence it is the
diagram of our partition $\lm$. Moreover the resulting is the regular
filling, and we are done. The case of the partition
$\lambda=(4,4,2,1)$ is displayed in Figure~\ref{f:figure-in-lemma}.
\end{proof}

\begin{figure}[h]
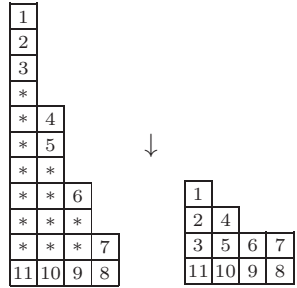

\[ \begin{array}{ccc}
	    \begin{matrix} 
	                 \Young{     
	                 \Box{1}      \cr
	                 \Box{2}      \cr
	                 \Box{3}      \cr
	                 \Box{*}      \cr
	                 \Box{*} & \Box{4}      \cr
	                 \Box{*} & \Box{5}      \cr
	                 \Box{*} & \Box{*}      \cr
	                 \Box{*} & \Box{*} & \Box{6}      \cr
	                 \Box{*} & \Box{*} & \Box{*}     \cr
	                 \Box{*} & \Box{*} & \Box{*} & \Box{7}      \cr
	                 \Box{11}  & \Box{10}  & \Box{9} & \Box{8}  \cr}
	      \end{matrix} 
& \downarrow &
	    \begin{matrix}
	               \Young{
	                \cr \; \cr \; \cr \; \cr \; \cr \; \cr \; \cr \; \cr \; \cr \; \cr \; \cr
	                \cr \; \cr \; \cr \; \cr \; \cr \; \cr \; \cr \; \cr \; \cr \; \cr \; \cr \; \cr \; \cr \; \cr \; \cr
	                \Box{1} \cr  
			\Box{2} & \Box{4}  \cr
	                \Box{3}& \Box{5} & \Box{6} & \Box{7}  \cr
	                \Box{11}& \Box{10} & \Box{9} & \Box{8}  \cr             
	                }
	     \end{matrix} 
\end{array}
\]
\caption{From the antidiagonal to the regular
filling.}\label{f:figure-in-lemma}
\end{figure}

\begin{remark}
Observe that $e_j(S)$ for $S$ of cardinality $j$ is a square free
monomial of degree $j$. So once we have all square-free monomials of
degree $n-\lm_1+1$ in our ideal, then we have the ones of higher
degree. These monomials are obtained when we read the generators
coming from the rightmost entry of the bottom row.
\end{remark}

The following statement follows easily from the previous remark and
Theorem~\ref{th:regular}.

\begin{corollary}[First reduction of Tanisaki's generating set for $\Il$]\label{t:first-reduction} 
Let $\lm$ be a partition of $n$. Then $\Il$ can be described as the
sum of the following three ideals:
$$\Il=\Ml+\El+\Kl,$$ where
\begin{itemize}
\item $\Ml$ is generated by all square-free monomials of degree
$n-\lm_1+1$;
\item $\El$ is generated by the elementary symmetric polynomials
$e_1(\xs),\ldots,e_{\ell(\lambda)-1}(\xs)$;
\item $\Kl$ is generated by the partially symmetric polynomials in
$e_r(k)$, where $n-1 \geq k \geq n-\lm_1+1$, and $r$ in an entry of
the regular filling of $\lambda$, in the same column as $k$, and
strictly above it.
\end{itemize}
\end{corollary}

In the particular case where the indexing partition $\lambda$ is a
hook, we recover the minimal generating set for $\mathcal{I}_\lambda$
described in \cite[Proposition 3.4]{BFR}.

\section{Second reduction of the generating set for $\Il$} \label{s:top-cells}

Our goal in the rest of the paper is to shave off as many redundant
generators as possible from the generating set given in
Corollary~\ref{t:first-reduction} . It turns out that only partially
symmetric polynomials coming from the top value of each column are
required in the generating set. This finding already gives a large
reduction in the number of generator needed in the generating set of
Tanisaki.  Several other reductions will be obtained in the
following sections.

Suppose we have a partition $\lm$ of an integer $n$, and fill the
diagram of $\lm$ with the regular filling defined in
Definition~\ref{d:regular-filling}. For $k\geq 1$ we label the value
in the top cell of the $k$-th column with $b_{k}$, as long as the
height of the $k$-th column is $\geq 2$.  If the right-most column of
$\lambda$ has height 1, then we label its entry $b_s$. This is
reflected in the diagram in Figure~\ref{f:reduction-figure}. Note that
with this notation we have
$$b_1=\lm'_1,\ 
b_2=\lm'_1+\lm'_2-1,\ \ldots,\ b_{k}=\lm'_1+\ldots+\lm'_k-k+1 \mbox{ for }
k \leq t, \ b_s=n-s, $$ where we set
\begin{equation}\label{def:st} 
t=\lm_2-1,\ \mbox{ and } \ s=\lm_1-1.
\end{equation}
Clearly if $\lambda_1=\lambda_2$, then $t=s$ and $b_s$ does not exist.

\begin{figure}[h]
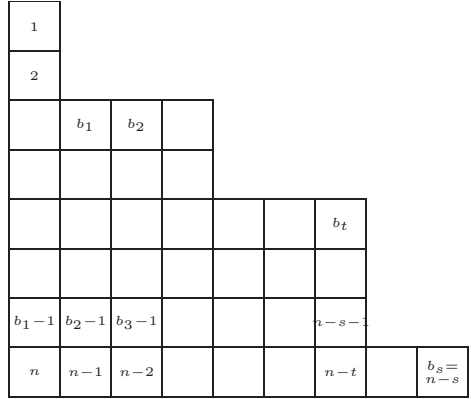

{\newdimen\Squaresize \Squaresize=19pt
\newdimen\Thickness \Thickness=.3pt
\newdimen\Correction \Correction=7pt
\[
\tiny{\begin{matrix} 
     \Young{ \Box{1}\cr
             \Box{2}\cr
             \Box{} &\Box{b_1} & \Box{b_2} &\Box{} \cr 
             \Box{} &\Box{} & \Box{} &\Box{} \cr 
             \Box{} &\Box{} & \Box{} &\Box{} &\Box{} &\Box{} &\Box{b_{t}}\cr 
             \Box{} &\Box{} & \Box{} &\Box{} & \Box{} &\Box{} &\Box{} \cr 
             \Box{b_1-1}&\Box{b_2-1} & \Box{b_3-1} &\Box{} & \Box{} &\Box{} & 
               \Box{n-s-1} \cr 
             \Box{n}&\Box{n-1} & \Box{n-2} &\Box{} & \Box{} &\Box{}& 
              \Box{n-t} &\Box{} & \Box{\mystack{b_s=}{n-s}}  \cr} 
     \end{matrix}}
\]
\caption{Diagram of a partition $\lm$ of $n$ with the regular
filling.}\label{f:reduction-figure}}
\end{figure}

By Corollary~\ref{t:first-reduction}) the reduced form of Tanisaki's
generating set for $\Il$ is the union of the following sets:

\begin{eqnarray}\label{e:generators}
\begin{tabular}{ll}
Column $0$ & $e_1(n),\ldots,e_{b_1-1}(n)$ \\
Column $1$ & $e_{b_1}(n-1),\ldots,e_{b_2-1}(n-1)$ \\ 
Column $2$ & $e_{b_2}(n-2),\ldots,e_{b_3-1}(n-2)$ \\
$\vdots$ & $\vdots$ \\
Column $t$ & $e_{b_{t}}(n-t),\ldots,e_{n-s-1}(n-t)$ \\
Column $s$ (if $s>t$) & $e_{n-s}(n-s)$, or all square-free monomials of degree $(n-s)$. 
\end{tabular}
\end{eqnarray}

Our goal here is to show that it is enough to pick only one set of
generators in each column, other than the $0$-th column; namely, the
ones coming from the top values in each column.

\begin{theorem}[Principal reduction of the generating set for $\Il$]\label{t:second-reduction} 
Let $\lm$ be a partition of $n$, and suppose that
the diagram of $\lm$ has been filled as in Figure~\ref{f:reduction-figure}. 
Then a generating set for $\Il$ is

\begin{eqnarray}\label{e:reduced-generators}
\begin{tabular}{ll}
Column $0$ & $e_1(n),\ldots,e_{b_1-1}(n)$  \\
Column $1$ & $e_{b_1}(n-1)$  (or $x_1^{b_1}, \ldots, x_n^{b_1}$)\\ 
Column $2$ & $e_{b_2}(n-2)$ \\
$\vdots$ & $\vdots$  \\
Column $t$ & $e_{b_{t}}(n-t)$  \\
Last column (if $s> t$) & $e_{n-s}(n-s)$, or all square-free monomials of degree 
  $(n-s)$. 
\end{tabular}
\end{eqnarray}

If $\lm=(1^n)$ is the one-column partition, then we also need to add
 the element $e_n(n)=x_1\cdots x_n$ to this generating set. If
$\lm=(n)$ is the one-row partition, we only need generators from the
last column, in other words $\I_{(n)}=(\xs)$.
\end{theorem}

    \begin{proof}  We need to show that having in the ideal all
     generators read from the top index of each column implies that
     the other partially symmetric functions coming from the larger
     indices in that column also belong to the ideal.  We go column by
     column, and build a new ideal $I_\lambda$ by adding generators
     described in (\ref{e:reduced-generators}) for each column of
     $\lambda$. We show, each time, that $I_\lambda$ contains all the
     other generators described in (\ref{e:generators}) (coming from
     the same column), and therefore $I_\lambda=\Il$.

    \begin{itemize}

     \item[\emph{Col. 0.}] There is nothing to prove here, as we are
     keeping all the generators $e_1(n), \ldots, e_{b_1-1}(n)$.

     \item[\emph{Col. 1.}] Assume that we have $e_{b_1}(S)\in I_\lambda$ for
     all $S$ with $|S|=n-1$. By Part 2 of Lemma~\ref{l:basic-lemma},
     setting $j=b_1$, we see that we have $e_{b_1}(n)\in I_\lambda$.

     For each $i>b_1$, we can assume by induction on $i$ that
     $$e_1(n),\ldots, e_{i-1}(n) \in I_\lambda \mbox{ and } \hfill
     e_{b_1}(n-1), \ldots, e_{i-1}(n-1) \in I_\lambda.$$ Apply
     Part 3 of Lemma~\ref{l:basic-lemma} with $j=i$, to see that
     $e_{i}(n) \in I_\lambda$.

     Fix a set $S$ with $|S|=n-1$ and $x \notin S$. Let $S^x=S\cup
     \{x\}$. Part 1 of Lemma~\ref{l:basic-lemma} implies that
     $$e_{i}(S)=e_{i}(S^x)-xe_{i-1}(S)$$ which demonstrates that
     $e_{i}(S) \in I_\lambda$. Hence $e_i(n-1) \in I_\lambda$.

      The fact that the generators $e_{b_1}(n-1)$ can be replaced by
      the powers $x_1^{b_1},\ldots, x_n^{b_1}$ follows directly from
      Proposition~\ref{p:new-presentation}.  Note that, in particular,
      we have $e_i(n-1) \in I_\lambda$, for all $i\geq b_1$.

    \item [\emph{Col. j.}]  Suppose $I_\lambda$ contains all generators
    from the previous columns $0,\ldots,j-1$ as described in
    (\ref{e:reduced-generators}). Let $|S|=n-j$, and suppose $x\notin
    S$, so that $|S^x|=n-j+1$, ($S^x=S\cup \{x\}$). We know by
    induction that $I_\lambda$ contains $e_h(S^x)$ for all $h\geq
    b_{j-1}$. Therefore, since $b_{j}>b_{j-1}$, for $i\geq b_{j}$ we
    have by Part 1 of Lemma~\ref{l:basic-lemma}
     
          $\begin{array}{lll}
        e_i(S)&=e_i(S^x)-xe_{i-1}(S)&=-xe_{i-1}(S)\\
        &=-x(e_{i-1}(S^x)-xe_{i-2}(S))&=x^2e_{i-2}(S)\\
        &=x^2(e_{i-2}(S^x)-xe_{i-3}(S))&=-x^3e_{i-3}(S)\\ &
        \hspace{.2in}\vdots & \\
        &=(-1)^{i-b_{j}}x^{i-b_{j}}e_{b_{j}}(S)& (mod \ Col.\ j-1)
      	\end{array}$

   This means that once we include $e_{b_{j}}(S)$ in $I_\lambda$, we
   will have  all $e_{i}(S) \in I_\lambda$ for $i\geq b_{j}$.
   
    \end{itemize} 
    \end{proof}

In the case where $\lm$ is a hook, the generating set described in
Theorem~\ref{t:second-reduction} coincides with the minimal
generating set for $\Il$ introduced in our earlier work~\cite{BFR}.

\begin{example}\label{e:first-reduction}  Let $\lm=(5,4,4,3)$. 
Then,  the regular filling of $\lambda$ is
\newdimen\Squaresize \Squaresize=12pt
\newdimen\Thickness \Thickness=.3pt
\newdimen\Correction \Correction=7pt
\begin{eqnarray*}
\begin{matrix} 
     \Young{ \Box{1} &\Box{4} & \Box{7} \cr 
             \Box{2} &\Box{5} & \Box{8} &\Box{10} \cr 
             \Box{3} &\Box{6} & \Box{9} &\Box{11} \cr 
             \Box{16}&\Box{15} &\Box{14} & \Box{13} &\Box{12} 
               \cr} 
     \end{matrix}
\end{eqnarray*}
So the generators of $\Il$ are
\begin{eqnarray*}
\begin{tabular}{c|l|l}
{\bf Column}& {\bf Generators} & {\bf Number of generators}\\
\hline&&\\ 
$0$ & $e_1(16),e_2(16),e_3(16) $ &3 \\
$1$ & $x_1^4,\ldots,x_{16}^4$ & 16\\
$2$ & $e_7(14)$&120\\
$3$ & $e_{10}(13)$&560\\
$4$ & $e_{12}(12)$, or all square-free monomials of degree $12$& 1820\\
\hline&&\\
&{\bf Total}& 2519
\end{tabular}
\end{eqnarray*}
Later in Example~\ref{e:final-reduction} we shall further reduce the
generating set of this particular partition.
\end{example}
\smallskip

\subsection{Remarks on a related work and conjecture of Weyman}\label{r:weyman-comparison}

\idiot{(sara) Weyman told me that he doesn't know if his conjectured
  minimal generating set is even a generating set.}

We end this section by showing some relations between the generating
set of Theorem~\ref{t:second-reduction} and two generating sets for $\Il$
arising in the work of Weyman~\cite{W1}.

In~\cite{W1} Weyman uses the representation theory of the general
linear group to construct and study generating sets for the ideal
$\Jl$ of polynomial functions vanishing on the conjugacy class
$\mathcal{C}_\lambda$. 
The generators in the first family, denoted by $V_{\lambda}$, are
expressed as  sums of minors, and
come from reducible representations of $GL(n)$.  The second set of
generators $U_{\lambda}$, on the other hand, arises from the
irreducible representations of $GL(n)$. The set $U_{\lambda}$  is 
smaller than $V_{\lambda}$, but
how to compute its elements is not explicit in the paper.
\smallskip

The set $V_\lambda$ (respectively $U_\lambda$) is given by the
disjoint union of sets $V_{i,p}$ (respectively $U_{i,p}$), where the
family of indices $(i,p)$ can be read off from a special diagram
introduced by Weyman; see~\cite[Example (4.5)]{W1}. We call this
diagram the {\em Weyman diagram} of $\lm$.  It is possible to
construct the Weyman diagram of a partition starting from the
antidiagonal filling (see Definition~\ref{d:antidiagonal-filling}) as
follows. First, consider the antidiagonal filling of
$\delta'(\lambda)$, and justify its columns in such a way that equal
entries are now in same rows. Then, replace any entry of this diagram
by an $X$. The resulting picture is the Weyman diagram. In
Figure~\ref{Weymandiagram} we illustrate the Weyman diagram
corresponding to the partition $\lambda=(4,4,2,1)$. Compare this
diagram to the one in Figure \ref{f:antidiagonal-filling}. Note that
if the top $X$ in the $i$-th column of Weyman diagram of $\lambda$ has
coordinates $(i,p)$, then the top cell of the $i$-th column of the
regular filling of $\lambda$ is filled by $p$.

\begin{figure}[h]
\begin{eqnarray*}
\begin{matrix} 
p= 1 & \underline{X} & & & & \\
p= 2 & \underline{X} & & & & \\    
p= 3 & \underline{X} & & & & \\
p= 4 & \underline{X} & \underline{X} & & & \\
p= 5 & \underline{X} & X & & & \\
p= 6 & \underline{X} & X & \underline{X} & & \\
p= 7 & \underline{X} & X & X & \underline{X} & \\
p= 8 & \underline{X} & X & X & X & \\
p= 9 & \underline{X} & X & X & & \\
p= 10 & \underline{X} & X & & & \\
p= 11 & \underline{X} &   & & & \\ 
i= & 0 & 1 & 2 & 3 \\
\end{matrix}
\end{eqnarray*}
\caption{Weyman diagram for $\lambda=(4,4,2,1)$.}\label{Weymandiagram}
\end{figure}

We would like to remark that Weyman follows a convention opposite to
ours when labelling the ideals $\Il$ and $\Jl$: he labels
$\mathcal{J}_\lambda$ the ideal of polynomial functions vanishing on
all nilpotent matrices with Jordan blocks $\lambda_1, \ldots,
\lambda_n$, while we use the transpose.  On the other hand, he
associates to a partition $\lambda$ what in our setting would be the
Weyman diagram of $\lambda'$.  These two facts cancel out, and we do
not need to take any transpose when reading statements involving his
diagrams.

\begin{definition}[Weyman's generating set for $\mathcal{J}_\lambda$]
 
In \cite[Theorem (4.6)]{W1} Weyman shows that the ideal $\Jl$ is
generated by the $U_{i,p}$, where the $(i,p)$'s are the coordinates of
the top cells of the columns ($i\geq 1$) of the Weyman diagram of
$\lambda$, together with the invariants $U_{0,p}$ with $1 \le p \le
n$. This result implies that the ideal $\Jl$ is also generated by the
$V_{i,p}$ coming from the same set of indices $(i,p)$. 
\end{definition}

\begin{example} For the partition $\lambda=(4,4,2,1)$, whose
Weyman diagram is in Figure~\ref{Weymandiagram}, Weyman's set
$U_\lambda$ consists of $U_{0,p}$, with $1 \le p \le 11$, $U_{1,4}$,
$U_{2,6}$, and $U_{3,7}$ (and similarly for the set $V_\lambda$). The
cells $X$ whose coordinates label this generating set are underlined.
\end{example}

After adding the generators for the ideal defining the diagonal
matrices to the two sets $V_\lambda$ and $U_\lambda$, one gets two
generating sets for $\Il$; we denote these two generating sets by
$\tilde{V_\lambda}$ and $\tilde{U_\lambda}$.

Instead of going into the definitions of $V_\lambda$ and $U_\lambda$
that can be found in \cite[Section 4]{W1}, we explicitly state the
cardinalities of their components in order to compare them with our
generating set. We emphasize the fact that Tanisaki's generators (the
ones we use) are easier to handle than Weyman's generators.  We have
that
$$|V_{i,p}|={ n \choose i}^2 \ \text{and} \ \ |\tilde{V}_{i,p}|={n
\choose i},$$ and
$$|U_{i,p}|={ n \choose i}^2 - { n \choose i-1}^2 \ \text{and}
\ \ |\tilde{U}_{i,p}|={n \choose i} - { n \choose i-1}.$$

\idiot{${ n \choose i} > { n \choose i-1}$ because of the diagonal and
  antidiagonal rules to construct Weyman's diagrams imply that $n> 2
  i$ (and $n > 2( n-i )$). Hence it follows from the fact that the
  strictly binomial coefficients ${n \choose k}$ increase until $k$ is half
  of $n$.}

 It turns out that the cardinalities of the generating set for $\Il$
 given by the $\tilde{V}_{i,p}$'s and the generating set given in
 Theorem~\ref{t:second-reduction} are the same. Moreover, it is not
 difficult to describe a one-to-one correspondence between the two
 generating sets. Under this correspondence Weyman's $V_{i,p}$
 generators correspond to our generators read from the top cell of the
 $i$-th column of the regular filling, as described in
 Theorem~\ref{t:second-reduction}.

Weyman conjectured that a special subset of $U_\lambda$ gives a
minimal generating set of $\Jl$; see Conjecture 5.1 and Remark 5.3 of
\cite{W1}.

\begin{conjecture}[Weyman's original conjecture]\label{c:weyman}
Let $\lambda$ be a partition. The set consisting of $U_{0,p}$ for
$1\leq p \leq \ell(\lambda)$, and $U_{i,p}$, where $(i,p)$ labels a
top cell of the $i$-th row (in the Weyman diagram of $\lambda$), such
that there are no $X$'s to the right of or on the line segment joining
$(i,p)$ with $(0,1)$, is a minimal set of generators
$\mathcal{W}_\lambda$ of $\Jl$.
\end{conjecture}

A very interesting question is the following.
\begin{question}[Diagonal version of Weyman's
    conjecture]\label{diagonalconjecture} Is the generating set
    $\tilde{\mathcal{W}}_\lambda$ for $\Il$ arising from Weyman's
    conjecture minimal ?
\end{question} 

In the following sections we show that the the answer to this question
is negative. Indeed, we provide some infinite families of
counterexamples. These observations, together with the help of
Macaulay 2 led us to the discovery that even the original conjecture
of Weyman (Conjecture~\ref{c:weyman}) fails already for one of the
smallest elements in these families.

\section{Reducing generators of $\Il$ of a fixed degree} \label{s:towardsU}

The aim of this section is to consider the generating set of $\Il$
described in Theorem~\ref{t:second-reduction}, and eliminate as many
redundant generators as possible from each column. 

\begin{proposition}[Columns of height $>1$]\label{p:reduction} 
Let $\lm$ be a partition whose diagram is represented in
Figure~\ref{f:reduction-figure}. For $k\ge2$, if the height of the
$(k-1)$-st column is $>1$, then we can eliminate ${n-1\choose k-1}+1$
generators of $\Il$ (as described in (\ref{e:reduced-generators}))
that come from the $k$-th column. Indeed, if $S$ denotes the set of
variables $x_1,\ldots,x_n$, we can eliminate the elements in the
set $\{e_{b_k}(S_{1,i_2,\ldots,i_{k}}) \mid 1< i_2 <\ldots<i_{k} \leq
n \}$ and $e_{b_k}(S_{2,3,\ldots,k+1})$.
\end{proposition}
 
\begin{proof} 
Let $k>1$, by using Part 2 of Lemma~\ref{l:basic-lemma} we write

       \begin{eqnarray}\label{e:Lemma-nilpotents-2} 
        \sum_{j\notin \{i_1,\ldots,i_{k-1}\}}e_{b_k}(S_{i_1, 
        \ldots,i_{k-1},j})= (n-b_k-k+1) e_{b_k}(S_{i_1,\ldots,i_{k-1}}) \equiv 0 \ ({\rm mod} \ \mathcal{I}_{k-1})
        \end{eqnarray}
where $\mathcal{I}_{k-1}$ is the ideal of generators coming from
columns $0$ to $k-1$.

So we have a system of $n \choose k-1$ linear homogeneous equations,
in $n \choose k$ variables.  In fact we have one equation for each
choice of a $(k-1)$-subset $\{i_1,\ldots,i_{k-1}\}$, and one variable
$e_{b_k}(S_{i_1, \ldots,i_{k-1},j})$ for each $k$-subset $\{i_1,
\ldots,i_{k-1},j\}$.
        
        The matrix associated to this system has columns $J$ indexed
        by the $k$-subsets of $\{1,2,\ldots,n\}$, and rows $I$ indexed
        by $k-1$-subsets of $\{1,2,\ldots,n\}$.  Equation
        (\ref{e:Lemma-nilpotents-2}), says that at position $(I,J)$
        the entry will be $1$ if $I \subseteq J$ and $0$ if $I
        \not\subseteq J$.
   
   We claim that we can drop from the generating set of
         Theorem~\ref{t:second-reduction} $ e_{b_k}(S_{J})$, for all
         $J$ of cardinality $k$ containing $1$, and $
         e_{b_k}(S_{2,\ldots,k+1}).$ To prove this it suffices to show
         that the submatrix corresponding to these columns has full
         rank $ {n-1 \choose k-1}+1$.
        
        We order the columns of this submatrix in this way: we put
        first the the columns indexed by a $J$ containing $1$ in
        alphabetical order, and then column indexed by
        $\{2,\ldots,k+1\}$. Similarly, we order the rows starting with
        those indexed by subsets $I $ that do not contain $1$, in
        alphabetical order, and then the row indexed by
        $\{1,\ldots,k-1\}$, and then the other rows in any order. In
        Figure~\ref{gauss} two examples are displayed.

        The square submatrix given by the first ${n-1 \choose k-1}+1$
        rows consists of two blocks. An identity $ {n-1 \choose
          k-1}$-matrix together with an additional row: $(1, \ldots,
        1, 0, \ldots, 0 )$, with $n-k+1$ ones. In fact, this last row is
        indexed by $\{1,\ldots,k-1\}$, and the entries are $1$ at
        columns indexed by $\{1,2,\ldots,k-1, j\}$ for $j>k$, and zero
        otherwise. By Gauss elimination, it is easy to see that this
        submatrix has full rank.
\end{proof}
 
\begin{figure}
$$\begin{array}{c|cccccc}
  & 12 & 13 & 14 & 23 & 24 & 34 \\
\hline
2 & {\bf 1} & {\bf 0} & {\bf 0} & {\bf 1} & 1 & 0 \\ 
3 & {\bf 0} & {\bf 1} & {\bf 0} & {\bf 1} & 0 & 1 \\ 
4 & {\bf 0} & {\bf 0} & {\bf 1} & {\bf 0} & 1 & 1 \\ 
1 & {\bf 1} & {\bf 1} & {\bf 1} & {\bf 0} & 0 & 0 \\ 
\end{array} 
\hspace{2cm}
\begin{array}{c|cccccc}
  & 123 & 124 & 134 & 234  \\
\hline
23 & {\bf 1} & {\bf 0} & {\bf 0} & {\bf 1}  \\ 
24 & {\bf 0} & {\bf 1} & {\bf 0} & {\bf 1}  \\ 
34 & {\bf 0} & {\bf 0} & {\bf 1} & {\bf 1}  \\ 
12 & {\bf 1} & {\bf 1} & {\bf 0} & {\bf 0}  \\ 
\hline
13 & {\bf 1} & {\bf 0} & {\bf 1} & {\bf 0}  \\ 
14 & {\bf 0} & {\bf 1} & {\bf 1} & {\bf 0}  \\ 
\end{array} 
$$
\caption{The non-singular submatrices for $n=4$, $k=2$, and $n=4$, $k=3$.}\label{gauss}  
\end{figure}

\begin{remark}  
	The system (\ref{e:Lemma-nilpotents-2}) has ${n \choose k-1}$
        linear equations and ${n \choose k}$ variables.  If all the
        equations are independent, then ${n \choose k-1}$ variables
        are redundant.  Hence only ${n \choose k}-{n \choose k-1}$ of
        them are necessary. Then using Gauss elimination we would
        obtain an explicit generating set of the same size as Weyman's
        $\tilde{U}_{k,p}$.  We note that there is no explicit
        construction for the generators in $U_\lm$ in Weyman's paper
        \cite{W1}.
 \end{remark} 

\begin{remark} 
	Let $\lambda$ be a partition of $n$ different than $(n)$.  As
        a consequence of Proposition~\ref{p:reduction}, the number of
        generators coming from the top cell of column $k$ in our
        generating set for $\Il$ is ${n \choose k} - {n-1 \choose k-1}
        - 1$.  On the other hand, and as discussed in
        Section~\ref{r:weyman-comparison} the corresponding
        $\tilde{U}_{k,p}$ in Weyman's generating set consists of ${n
          \choose k} - {n \choose k-1}$ elements. Since for all
        partitions other than $(n)$, we have that $n > k$, we conclude
        that the difference between the two sets is ${ n-1 \choose k-2
        } - 1 $, for each $k > 2$. For columns $0$, $1$, and $2$ their
        cardinalities coincide.
\end{remark}

 We now focus on eliminating generators from a column of height 1. 

\begin{proposition}[Columns of height 1]\label{p:reduction-height-1} 
Let $\lm$ be a diagram represented in
Figure~\ref{f:reduction-figure}. If $s> t \geq 1$, then we can
eliminate ${n-s+t \choose t}$ square-free monomial generators of $\Il$
coming from the last column.\end{proposition}

 \begin{proof} Note that as $n-s>b_t$ (see Figure~\ref{f:reduction-figure}), from the proof of Theorem~\ref{t:second-reduction} we know that
      $e_{n-s}(n-t) \in \Il$. We now claim that 
      we can drop monomial generators of the form      
        $$\begin{array}{ll} 
        e_{n-s}(S_{1,2,\ldots,s-t,i_1,\ldots,i_{t}}),& 
        \    s-t< i_1< i_2< \ldots < i_t \leq n
	  \end{array}$$ from the generating set for $\Il$.
        Since there are ${n-s+t \choose t}$ such choices for sets
        $\{i_1,\ldots,i_t\}$, this will settle the statement of the
        proposition. But this follows from the trivial identity
        \[
        e_k(A) =  \sum_{\substack{ J \subseteq A\\ |J|=k}} e_k(J),
        \]
which implies 
        \[
        e_{n-s}(S_{1,2,\ldots,s-t,i_1,\ldots,i_{t}})=e_{n-s}(S_{i_1,\ldots,
         i_{t}})-\sum_{\mysumstack{\mbox{$\scriptstyle
         \{j_1,\ldots, j_{s-t}\} \neq
         \{1,\ldots,s-t\}$}}{\mbox{$\scriptstyle \{j_1,\ldots,
         j_{s-t}\}\cap \{i_1,\ldots,i_t\} =\emptyset$}}}
         e_{n-s}(S_{j_1,\ldots, j_{s-t}, i_1,\ldots,i_{t}}) \in \Il.
         \]
      \end{proof}

Therefore using Propositions~\ref{p:reduction}
and~\ref{p:reduction-height-1}, we have reduced our generating set to
that in the table in Figure~\ref{f:table}, using the Vandermonde
identity ${n \choose k} = {n-1 \choose k-1}+{n-1\choose k}$.

\begin{figure}[h]
$$
\begin{array}{l|l|l}
{\bf Column} & {\bf Generators} & {\bf Number}\\
\hline&&\\
0 & e_1(n),\ldots,e_{b_1-1}(n) &  b_1-1=\lm'_1-1 \\
1 & x_1^{b_1}, \ldots, x_n^{b_1} & {n\choose 1}\hfill ={n-1\choose 1}+1\\ 
2 & e_{b_2}(n-2) & {n\choose 2} - {n-1\choose 1}-1 \hfill ={n-1\choose 2}-1\\
\vdots & \hspace{.25in} \vdots & \hspace{.35in} \vdots \\
t & e_{b_{t}}(n-t) & {n\choose t} - {n-1\choose t-1}-1  
                         \hfill ={n-1\choose t}-1 \\
s \ (\mbox{if } s>t) &e_{n-s}(n-s)& {n \choose s}-{n-s+t \choose t}
\end{array}
$$\caption{Number of generators in each degree in the reduced generating set for $\Il$}\label{f:table}
\end{figure}

\begin{example}
Consider the partition $\lm=(4,4,2,1)$ in Figure~\ref{f:regular-filling}.
Our formula gives 177 generators, but in fact, Macaulay2 verifies that
168 generators are enough. The extra generators are in degree 7 (see
table in Figure~\ref{f:table}):

\bigskip
\begin{tabular}{l|l|l}
{\bf Degrees} & {\bf Number of generators from Table~\ref{f:table}} &\bf{Actual number of generators required}\\
\hline &&\\
1, 2, 3 & 1 in each degree & 1 in each degree\\
4 &11&11\\
6&44&44\\
7&119&110 
\end{tabular}
\bigskip
\end{example}

While in many examples such as the previous one, the predictions of
the diagonal version of Weyman's conjecture are correct, this is not
always the case.

\begin{example}\label{e:counterexample} 
Consider the partition $\lambda=(5,4,1)$. 
\newdimen\Squaresize \Squaresize=12pt
\newdimen\Thickness \Thickness=.3pt
\newdimen\Correction \Correction=7pt
\begin{figure}[h]
\begin{eqnarray*}
	           \Young{     
	            \Box{1}   \cr
	             \Box{2} & \Box{3} & \Box{4} & \Box{5}   \cr
	             \Box{10} & \Box{9} & \Box{8} &\Box{7} & \Box{6}  \cr}
\end{eqnarray*}
\caption{The partition $\lambda=(5,4,1)$}
\end{figure}
We denote by $\mathcal{I}_{01}=(e_1(10),e_{2}(10), x_1^3,\ldots,
x_{10}^3)$ the ideal generated by the elements of the $0$-th and
$1$-st column. Now consider $e_4(8)$ coming from the second
column. Let $A \subseteq \{1, \ldots, n \}$ be a subset of of
cardinality $8$, and let $B$ be its complement ($|B|=2$). By
Proposition~\ref{p:new-presentation}, we have mod
$\E_3(10)$ \begin{eqnarray}\label{ex:541} e_4(A)\equiv h_4(B) =
  m_{(4)}(B)+m_{(3,1)}(B)+m_{(2,2)}(B).
\end{eqnarray}

Among the monomial symmetric polynomials appearing in (\ref{ex:541}), $m_{(4)}$, and $m_{(3,1)}$ are already in the $\I_{01}$, since it contains $x_1^3,\ldots, x_n^3.$
So from the second column we only need to add the set $m_{(2,2)}(2)$ to the generators of $\mathcal{I}_{01}$ to obtain a bigger ideal denoted $\mathcal{I}_{012}$ included in $\Il$. That is, we need to add all generators of the form $(x_i x_j)^2$ for $i<j$.

Now let us consider $e_5(A)$, where $|A|=7$ and $B$ is its
complement. From the third column
\begin{multline}\label{3partitions}
- e_5(A) \equiv  h_5(B)  
          = m_{(5)}(B)+m_{(3,2)}(B)+m_{(4,1)}(B)+m_{(3,1,1)}(B)+m_{(2,2,1)}(B).
\end{multline}
It is clear that each one of these monomial symmetric polynomials is
already in the ideal $\I_{012}$. In fact, every monomial in the first
four summands in (\ref{3partitions}) contains a power $x_i^3$, and
each element in $m_{(2,2,1)}(B)$ can be obtained as a combination of
elements in $m_{(2,2)}(2)$.  Hence the third column will not
contribute any new generator. The same happens for the last
column. Let $|A|=6$ and $B$ be its complement, $|B|=4$.  Then
\begin{eqnarray*}
e_6(A) = h_6(B) &=& m_{(6)}(B)+ m_{(5,1)}(B) + m_{(4,2)}(B) + m_{(3,3)}(B)\\ 
&+& m_{(4,1,1)}(B) + m_{(3,2,1)}(B)+m_{(2,2,2)}(B) \\
&+& m_{(3,1,1,1)}(B)+m_{(2,2,1,1)}(B),
\end{eqnarray*}
and all monomials in this sum are already in the ideal, since they
contain either a power $x_i^3$, or a monomial $(x_ix_j)^2$. So we have
$\Il=\I_{012}$.
\end{example}

\idiot{(Sara) According to an old M2 file that I have, the betti diagram for a minimal generating set
  for $\I_{(5,4,1)}$ is

\begin{tabular}{rrr}
           & 0&  1\\
o5 = total:& 1& 47\\
         0:& 1&  1\\
         1:& .&  1\\
         2:& .& 10\\
         3:& .& 35\\
\end{tabular}

and for $\I_{(5,4,1)}$ itself is
\bigskip

\begin{tabular}{rrr}
           & 0&   1 \\
o3 = total:& 1& 387\\
         0:& 1&   1\\
         1:& .&   1\\
         2:& .&  10\\
         3:& .&  45\\
         4:& .& 120\\
         5:& .& 210\\
           & 0&  1\\
\end{tabular}
}

\begin{counterexample}[Counterexample to the diagonal version of Weyman's conjecture]\label{primocontro}

Example~\ref{e:counterexample} proves that the generating set
$\tilde{\mathcal{W}}_\lambda$ for $\Il$ coming from the minimal
generating set for $\Jl$ conjectured by Weyman is not in general
minimal (see Question~\ref{diagonalconjecture}). More precisely,
according to his diagram in Figure~\ref{We-diag}, some generators of
degree $5$ and $6$ should be needed, while they are not, as we just
showed. In Figure~\ref{We-diag} the coordinates of the underlined
$X$'s label the generators of $\Il$ arising from the diagonal version
of Weyman's conjecture. The generators coming from the shaded $X$'s are
not needed. This is the convention that we shall use later as well.

\begin{figure}[h]
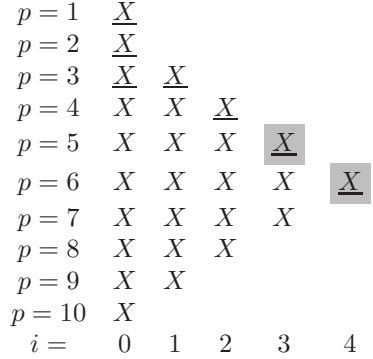

\begin{eqnarray*}
\begin{matrix} 
p= 1 & \underline{X} & & & & \\
p= 2 & \underline{X} & & & & \\    
p= 3 & \underline{X} & \underline{X} & & & \\
p= 4 & X & X & \underline{X} & & \\
p= 5 & X & X & X & \shadebox{\underline{X}} & \\
p= 6 & X & X & X & X & \shadebox{\underline{X}} \\
p= 7 & X & X & X & X &  \\
p= 8 & X & X & X &  & \\
p= 9 & X & X &  & & \\
p= 10 & X &  & & & \\
i= & 0 & 1 & 2 & 3 & 4\\
\end{matrix}
\end{eqnarray*}
\caption{Weyman diagram for $\lambda=(5,4,1)$.}\label{We-diag}  
\end{figure}
\end{counterexample}

It might be possible to generalize the reasoning used in
Example~\ref{e:counterexample} with an algorithm, as explained below.

\begin{algorithm}\label{thealgorithm}
Consider the Young diagram of $\lambda$ filled with the regular
filling. Let $b_1,\ldots,b_s$ be the top-cell entries of $\lambda$ as in
Figure~\ref{f:reduction-figure}. Set
$\mathcal{G}_0=\{e_1(n), \ldots, e_{b_1-1}(n)\}$, and create a list of
partitions $L_0=\emptyset.$ For all $k\geq 1$, define
$$ U_k = \{ \mu \vdash b_k \,|\, \ell(\mu) \leq k \ \text{ and} \ \nu
\not \subseteq \mu, \ \mbox{for any} \ \nu \in L_{k-1} \},
$$ where $\nu \subseteq \mu$ means that the Young diagram of $\nu$ is
contained in that of $\mu$.

\begin{itemize}
\item[1)] If $|U_k|=1$, say $U_k=\{ \theta \}$, then $L_k = L_{k-1} \cup
\{ \theta \}$ and $\mathcal{G}_k = \mathcal{G}_{k-1} \cup
m_{\theta}(k)$.

\item[2)] If $|U_k|=0$, then $\mathcal{G}_k=\mathcal{G}_{k-1}$ and  $L_k=L_{k-1}$.

\item[3)] If $|U_k|>1$, then $\mathcal{G}_k = \mathcal{G}_{k-1} \bigcup
\big( \bigcup_{l\geq k} h_{b_l}(l) \big)$, and stop.
\end{itemize}
Denote by $\mathcal{G}$ the set produced by the algorithm at the last step.
\end{algorithm}

\begin{question}
Is the  set $\mathcal{G}$ a generating set for $\Il$?
\end{question}

Clearly this algorithm produces a subset of the generating
set given by the Theorem~\ref{t:second-reduction}. All generators
coming from cells labeled $b_k$ satisfying condition $2)$ in the above
algorithm would become redundant.

We used this algorithm to produce generating sets for all families of
examples and counterexamples considered in the next section. Then, we
proceeded to prove their correctness on a one by one basis. A proof of
the correctness of the algorithm would be greatly welcomed.

\section{Families of examples and a counterexample to Weyman's conjecture}\label{ultima}

We conclude the paper by producing simple generating sets for some
particular families of shapes. In particular, this allows us to
construct two infinite families of counterexamples to the diagonal
version of Weyman's conjecture (Question~\ref{diagonalconjecture}), as
well as a counterexample to the original conjecture of Weyman for a
minimal generating set of the ideal $\mathcal{J}_\lambda$ (see
Conjecture~\ref{c:weyman}).


\begin{example}[The case of  two-column partitions]
As mentioned above a partition of $n$ of the form $\lm=(2^a, 1^c)$,
where $a+c=\ell=\ell(\lambda)$ the length of the partition, $\Il$ is
generated by $e_1(n),\ldots,e_{\ell-1}(n),$ $x_1^{\ell}, \ldots,
x_{n}^{\ell}$.
\end{example}

\begin{theorem}[The case of partially-rectangular partitions]
\label{t:reduction2} 
Let $\lm$ be a partition of $n$, and let $k>2$ be any integer.  If
columns $0,1,\ldots, k-1$ of the Young diagram have the same height,
then in the generating set for the ideal $\Il$ described in
Theorem~\ref{t:second-reduction} generators coming from columns
$2,\ldots,k$ are redundant.
\end{theorem}

\begin{proof}  The regular filling of the partition $\lm$ has the 
following form.  

\newdimen\Squaresize \Squaresize=20pt
\newdimen\Thickness \Thickness=.3pt \newdimen\Correction
\Correction=5pt
\begin{eqnarray*}
\begin{matrix} 
     \Young{ \Box{1} &\Box{g+1} & \Box{2g+1} &\Box{\cdots}  &\cr 
                    \Box{2} &\Box{g+2} & \Box{2g+2} &\Box{\cdots} &  \Box{kg+1}  \cr
                    \Box{\cdots} &\Box{\cdots} & \Box{\cdots} &\Box{\cdots} &  \Box{\cdots}&  \Box{\cdots} \cr
                    \Box{g} &\Box{2g} & \Box{3g} &\Box{\cdots} &\Box{\cdots} & \Box{\cdots} \cr 
                    \Box{n} &\Box{\cdots} & \Box{\cdots} &\Box{\cdots} &\Box{\cdots} & \Box{\cdots} \cr 
               \cr} 
     \end{matrix}
\end{eqnarray*}

        By Theorem~\ref{t:second-reduction} and
        Proposition~\ref{p:new-presentation}, modulo the previous
        columns, the generators coming from Column $k$ are of the form
        $$h_{kg+1}=\sum_{a_1+\ldots+a_k=kg+1}x_{j_1}^{a_1}\ldots
        x_{j_k}^{a_k}$$ where $ 1 \leq j_1 \leq \ldots \leq j_k \leq
        n$.

         Consider a term $x_{j_1}^{a_1}\ldots x_{j_k}^{a_k}$ in the
	 sum above. We claim that for at least one power $a_i$,
	 $a_i\geq g+1$, making this monomial redundant in the presence
	 of the second column generators, which are the $(g+1)$-st
	 powers of the variables.

	 To see this, suppose $a_1\leq g, \ldots, a_k\leq g$. Then we
	 should have that $$kg+1=a_1+\ldots+a_k\leq kg$$ which is a
	 contradiction.
\end{proof}

\begin{remark}  
Drawing the Weyman diagram associated to partially rectangular
partitions considered in Theorem~\ref{t:reduction2}, one can see that
the points $(0,1),$ $(1,g+1),$ $(2,2g+1), \ldots,(k, kg+1)$ are
collinear because they can successively obtained by adding the vector
$(1,g)$. Therefore, the diagonal version of Weyman's conjecture
predicts that the generators coming from cells $(2,2g+1), \ldots,
(k,kg+1)$ are redundant. This is true: in fact these are precisely the
redundant cells according to Theorem~\ref{t:reduction2}.
\end{remark}

\begin{example}\label{e:final-reduction} Let $\lm=(5,4,4,3)$ be the 
partition in
Example~\ref{e:first-reduction}. Theorem~\ref{t:reduction2} implies
that the generating set for $\Il$ consists of the elements in the
second column in the table below (compare with
Example~\ref{e:first-reduction}), and the reduced number from the
table in Figure~\ref{f:table} is in the third column. No $7$ and
$10$-degree generators are needed in the generating set.  In this case
the prediction of the diagonal version of Weyman's conjecture was
correct: cells $(2,7)$ and $(3,10)$ are redundant; see Figure
\ref{Weymandiagram3}.
\begin{center}
\begin{tabular}{c|l|l}
{\bf Column}& {\bf Generators} & {\bf Numbers from  Figure~\ref{f:table}}\\
\hline&&\\ 
 0 & $e_1(16),e_2(16),e_3(16)$&3\\
 1 & $x_1^4, \ldots, x_{16}^4$&16\\ 
 2 & \text{redundant} & --\\
 3 & \text{redundant}& --\\
 4 & $e_{12}(12)$ & 1365\\
\hline&&\\
&{\bf Total}& 1384 
\end{tabular}
\end{center}
\end{example}

\begin{figure}
\begin{eqnarray*}
\begin{matrix} 
 1 & \underline{X} & & & & \\
 2 & \underline{X} & & & & \\    
 3 & \underline{X} & & & & \\
 4 & \underline{X} & \underline{X} & & & \\
 5 & X & X & & & \\
 6 & X & X & & & \\
 7 & X & X & \shadebox{\underline{X}} &  & \\
 8 & X & X & X &  & \\
 9 & X & X & X & & \\
 10 & X & X & X & \shadebox{\underline{X}} & \\
 11 & X & X &X &X & \\ 
 12 & X & X &X &X & \underline{X} \\ 
 13 & X & X &X &X & \\
 14 & X & X & X &  & \\
 15 & X & X & & & \\
 16 & X &  & & & \\
i= & 0 & 1 & 2 & 3 & 4
    \end{matrix}
\end{eqnarray*}
\caption{An example of a partially--rectangular partition
  $\lambda=(5,4,4,3)$.}\label{Weymandiagram3}
\end{figure}

\begin{corollary}[The case of rectangular partitions]
For a rectangular partition of $n$ of the form $\lm=(u^\ell)$, the
generating set of $\Il$ will simply be $e_1(n),\ldots,e_{\ell-1}(n),
x_1^{\ell}, \ldots, x_{n}^{\ell}$, where $n=u\,\ell.$
\end{corollary}

\begin{corollary}[The case of two-row partitions]
 For a two-row partition of $n$ of the form $\lambda=(u,v)$, a generating set is given by
$e_1(n)$, $x_1^{2}, \ldots, x_{n}^{2}$, and $e_u(u)$.
\end{corollary}

\begin{theorem}\label{t:reduction2bis} Let $\lm$ be a partition of $n$.

\begin{enumerate}
\item If $\lm=(u^a,(u-1)^c)$ with $g=a+c$, then 
a generating set of $\Il$ is given by 
\[ e_1(n),\ldots,e_{g-1}(n), x_1^{g},\ldots, x_n^{g}.\]

\item If $\lambda=(u^a, (u-1)^c,1)$ with $u \ge3$ and $g=a+c>1$, then
$\Il$ is generated by
$$e_1(n), \ldots, e_g(n),x_1^{g+1}, \ldots, x_n^{g+1}, (x_1x_2)^g, (x_1x_3)^g, \ldots, (x_{n-1}x_n)^g.$$

\item If $\lambda=(u^a, (u-1)^c, 1,1)$ with $u \ge4$ and $g=a+c+1>2$,
then $\Il$ is generated by
 \begin{align*}
e_1(n), \ldots, e_g(n), x_1^{g+1}, \ldots, x_n^{g+1}, (x_i+x_j)(x_i  x_j)^{g-1} \text{ for all $ i \ne j $}, \text{and } (x_i x_j x_k)^{g-1} \text{ for all $i<j<k$}.
\end{align*}
\end{enumerate}
\end{theorem}

\begin{proof}
\begin{enumerate}
\item This is an easy consequence of Theorem~\ref{t:reduction2}.
 
\item The regular filling of $(u^a, (u-1)^c,1)$ will be of the form:

\newdimen\Squaresize \Squaresize=21pt \newdimen\Thickness
\Thickness=.3pt \newdimen\Correction \Correction=5pt

\begin{eqnarray*}
\begin{matrix} 
     \Young{ \Box{1} \cr 
                    \Box{2} &\Box{g+1} & \Box{2g} &\Box{3g-1} &  \Box{\cdots} & \Box{\mystack{lg-l}{+2}}  &  \Box{\cdots} \cr
                    \Box{\cdots} &\Box{\cdots} & \Box{\cdots} &\Box{\cdots} &  \Box{\cdots}&  \Box{\cdots}&  \Box{\cdots} & \cr
                                        \Box{\cdots} &\Box{\cdots} & \Box{\cdots} &\Box{\cdots} &  \Box{\cdots}&  \Box{\cdots}&  \Box{\cdots} & 
\Box{\mystack{(u-1)g}{-u+3}} \cr
                    \Box{g} &\Box{\cdots} & \Box{\cdots} &\Box{\cdots} &\Box{\cdots} & \Box{\cdots} &  \Box{\cdots} &  \Box{\cdots} \cr 
                    \Box{n} &\Box{n-1} & \Box{n-2} &\Box{n-3} &\Box{\cdots} & \Box{n-l}&  \Box{\cdots} &  \Box{\mystack{n-u}{+1}}  \cr 
               \cr} 
     \end{matrix}
\end{eqnarray*}  

Columns $0$ and $1$ clearly provide the generators $e_1(n), \ldots,
e_g(n),x_1^{g+1}, \ldots x_n^{g+1}$. By
Proposition~\ref{p:new-presentation}, Column $2$ provides generators
of the form $$h_{2g}=\sum_{a+b=2g}x_i^ax_j^b$$ for $1\leq i < j \leq
n$. Since we already have $x_i^{g+1}$ and $x_j^{g+1}$ in the ideal,
this sum reduces to the monomial $x_i^gx_j^g$. Hence the third column
provides the remaining generators $(x_1x_2)^g, (x_1x_3)^g, \ldots,
(x_{n-1}x_n)^g$.

It remains to show that the generators coming from Columns $3,\ldots,
u-1$ are redundant. Let $l$ be any integer such that $3 \leq l \leq
u-1$. The generators from Column $l$, by
Proposition~\ref{p:new-presentation} and the fact that we have all $(g+1)$-st powers of the variables in the ideal, are of the form
$$h_{lg-l+2}=\sum_{\mysumstack{a_1+\cdots+a_l=lg-l+2}{a_1,\ldots,a_l\leq
g}}x_{i_1}^{a_1}\ldots x_{i_l}^{a_l}$$ where $1\leq i_1< i_2< \ldots
<i_l\leq n$, and in each monomial $x_{i_1}^{a_1}\ldots
x_{i_l}^{a_l}$ at most one of the powers $a_u$ is equal to
$g$. For such a monomial in the sum, we therefore have
$$a_1+\cdots+a_l\leq (l-1)(g-1)+g=lg-l+1 \Longrightarrow lg-l+2 \leq
lg-l+1$$ which is a contradiction. So there is no generator from
Column $l$ if $l \geq 3$.

\item The regular filling of $(u^a, (u-1)^c, 1,1)$ will be of the
following form. \newdimen\Squaresize \Squaresize=21pt \newdimen\Thickness
\Thickness=.3pt \newdimen\Correction \Correction=5pt
\begin{eqnarray*}
\begin{matrix} 
     \Young{ \Box{1} \cr
                     \Box{2} \cr 
                    \Box{3} &\Box{g+1} & \Box{2g-1} &\Box{3g-3} &  \Box{\cdots} & \Box{\mystack{lg-2l}{\scriptscriptstyle{+3}}}  &  \Box{\cdots} \cr
                    \Box{\cdots} &\Box{\cdots} & \Box{\cdots} &\Box{\cdots} &  \Box{\cdots}&  \Box{\cdots}&  \Box{\cdots} \cr
                    \Box{g} &\Box{\cdots} & \Box{\cdots} &\Box{\cdots} &\Box{\cdots} & \Box{\cdots} &  \Box{\cdots} &\Box{\mystack{(u-1)g}{-2u+5}}\cr 
                    \Box{n} &\Box{\cdots} & \Box{\cdots} &\Box{\cdots} &\Box{\cdots} & \Box{\cdots}&  \Box{\cdots} &\Box{} \cr 
               \cr} 
     \end{matrix}
\end{eqnarray*}  

Again Columns $0$ and $1$ provide the generators $e_1(n), \ldots,
e_g(n),x_1^{g+1}, \ldots x_n^{g+1}$.

By Proposition~\ref{p:new-presentation}, Column $2$ provides
generators of the form $$h_{2g-1}=\sum_{a+b=2g-1}x_i^ax_j^b$$ for $1\leq i < j
\leq n$. Since we already have $x_i^{g+1}$ and $x_j^{g+1}$ in the
ideal, we can additionally assume that $a,b \leq g$ for each monomial
$x_i^ax_j^b$ in the sum, and so at least one of $a$ or $b$ would have
to be $g-1$ and the other $g$. This produces a generator of the form
$x_i^gx_j^{g-1}+ x_i^{g-1}x_j^g=(x_i+x_j)(x_i x_j)^{g-1}$.

Similarly, Column $3$ will produce generators of the form
$$h_{3g-3}=\sum_{a+b+c=3g-3}x_i^ax_j^bx_k^c$$ for $1\leq i < j <k \leq n$. Once more, we can assume that $a,b,c \leq g$, which reduces the sum above to
$$\begin{array}{l}
x_i^{g-1}x_j^{g-1}x_k^{g-1}+ 
x_i^{g-2}(x_j^{g}x_k^{g-1}+x_j^{g-1}x_k^{g})+ 
x_j^{g-2}(x_i^{g}x_k^{g-1}+x_i^{g-1}x_k^{g})+
x_k^{g-2}(x_i^{g}x_j^{g-1}+ x_i^{g-1}x_j^{g})\\
\\
= x_i^{g-1}x_j^{g-1}x_k^{g-1}+ 
x_i^{g-2}x_j^{g-1}x_k^{g-1}(x_j+ x_k)+
x_j^{g-2}x_i^{g-1}x_k^{g-1}(x_i+ x_k)+
x_k^{g-2}x_i^{g-1}x_j^{g-1}(x_i+ x_j).
  \end{array}
$$

The last three summands are in the ideal already (coming from Column
$2$), so the generators from Column $3$ can all be written as $
x_i^{g-1}x_j^{g-1}x_k^{g-1}$ for $1\leq i <j<k \leq n$.

We now need to show that generators coming from Column $l$, where $4
\leq l \leq u-1$ are redundant. The generators from Column $l$, by
Proposition~\ref{p:new-presentation} and the fact that we have all
$(g+1)$-st powers of the variables in the ideal, are of the form
$$h_{lg-2l+3}=\sum_{\mysumstack{a_1+\cdots+a_l=lg-2l+3}{a_1,\ldots,a_l\leq
g}}x_{i_1}^{a_1}\ldots x_{i_l}^{a_l}$$ where $1\leq i_1< i_2< \ldots
<i_l\leq n$.


Suppose that $M=x_{i_1}^{a_1}\ldots x_{i_l}^{a_l}$ is a monomial in this
sum.  

If one of the powers, say $a_1$, is equal to $g$, then we must have
another power among $a_2,\ldots, a_l$ that is $g$ or
$g-1$. If not, all of $a_2, \ldots,a_l$ are $\leq g-2$, and we have
$$lg-2l+3=a_1+\cdots+a_l\leq g+(l-1)(g-2)= lg-2l+2$$ which is a
contradiction. So there is at least another power, say $a_2$, such
that $a_2\geq g-1$. 

\begin{itemize} 

\item  $a_1=a_2=g$. In this case, we can write
$$\begin{array}{ll}
x_{i_1}^{g}x_{i_2}^{g}x_{i_3}^{a_3}x_{i_4}^{a_4}...x_{i_l}^{a_l}&=
(x_{i_1}+x_{i_2})(x_{i_1}x_{i_2})^{g-1} [ 1/2 x_{i_2}x_{i_3}^{a_3}x_{i_4}^{a_4}...x_{i_l}^{a_l}+ 1/2
x_{i_1}x_{i_3}^{a_3}x_{i_4}^{a_4}...x_{i_l}^{a_l}]\\
&-1/2x_{i_1}^{g+1}x_{i_2}^{g-1}x_{i_3}^{a_3}x_{i_4}^{a_4}...x_{i_l}^{a_l}
-1/2x_{i_1}^{g-1}x_{i_2}^{g+1}x_{i_3}^{a_3}x_{i_4}^{a_4}...x_{i_l}^{a_l}
  \end{array}$$

All the terms on the right-hand side are already in the ideal, and
hence so is
$x_{i_1}^{g}x_{i_2}^{g}x_{i_3}^{a_3}x_{i_4}^{a_4}...x_{i_l}^{a_l}$.

\item $a_1=g$ and $a_2=g-1$. In this case, there is another monomial
$M'=x_{i_1}^{g-1}x_{i_2}^{g}x_{i_3}^{a_3}x_{i_4}^{a_4}...x_{i_l}^{a_l}$
in the sum as well, and there is exactly one copy of $M$ and one copy
of $M'$ in the sum. Now we have
$$M+M'=(x_{i_1}+x_{i_2})x_{i_1}^{g-1}x_{i_2}^{g-1}(x_{i_3}^{a_3}x_{i_4}^{a_4}...x_{i_l}^{a_l}).$$

So each such monomial $M$ is paired with a unique monomial $M'$ in the
sum, and their sum is already in the ideal.
\end{itemize}

Now assume that all the powers  $a_1, \ldots,a_l$ are  $\leq g-1$.
If $l-2$ of the powers $a_1, \ldots,a_l$ are $\leq g-2$, then we
have $$lg-2l+3=a_1+\cdots+a_l\leq (l-2)(g-2) + 2(g-1)= lg-2l+2$$ which
is a contradiction. So there are at least 3 powers among $a_1,
\ldots,a_l$ that are equal to $g-1$. But then the monomial
$x_{i_1}^{a_1}\ldots x_{i_l}^{a_l}$ is already in $\Il$, because it is
a multiple of a generator coming from Column $3$.
\end{enumerate}
\end{proof}

\begin{corollary}
Suppose that the first $l+1$ columns of a partition $\lambda$ belong
to one of the three families of shapes described in
Theorem~\ref{t:reduction2bis}. Then
\begin{itemize}
\item [a)] In cases 1 and 2, the generators coming from Columns
  $3,\ldots, l$ are redundant. For Columns $0,1,2$ we can use the
  generators described in Theorem~\ref{t:reduction2bis}.

\item[b)] In Case 3, the generators coming from columns $4,\ldots, l$
are redundant. For Columns $0,1,2,3$ we can use the generators described
in Theorem~\ref{t:reduction2bis}.
\end{itemize}
\end{corollary}

\begin{counterexample}[Counterexamples to the diagonal version of Weyman's conjecture] \label{counterdiagonal}
The two infinite families of partitions described in parts 2 and 3 of
Theorem~\ref{t:reduction2bis} are counterexamples to the diagonal
version of Weyman's conjecture. Indeed, according to it, all
generators coming from each of the top cells of their diagrams should
be necessary because for $k>0$, the top cells are collinear (for the
first family we can move from one top cell to the next one by adding
the vector $(1,g-1)$, and for the second family, by adding the vector
$(1,g-2)$).  But the line containing those points does not pass
through $(0,1)$.  Instead it passes through $(0,2)$ for the first
family, and through $(0,3)$ for the second family.

Let $\lambda$ be a partition such that its first $l$ columns belong to
one of the two families of shapes described above, with $l > 2$ for
the first family and $l> 3$ for the second one. The preceding
corollary shows that the generators coming from Column $k$, with $3<k
\leq l$ are redundant. We conclude that each such $\lambda$ is a
counterexample to the diagonal version of Weyman's conjecture. A first
counterexample was shown in Counterexample~\ref{primocontro}.
\end{counterexample}

\todo{(sara) The following example said that ``Using Macaulay2 we see
  that the cell containing 7 is redundant.'' But this also comes from
  our theorem. I have mentioned it in the text. If you agree (please
  check for yourselves!) then please remove this box}

\begin{example}
Consider the partition $(5,5,1,1)$ that fits inside one of the
families in Theorem~\ref{t:reduction2bis}. As proved in that theorem,
the cell containing 7 is redundant. Translated into the Weyman
diagram, this means that the ${\underline{X}}$ in position $(4,7)$
is redundant (see Figure~\ref{figurequestion0}).
\begin{figure}
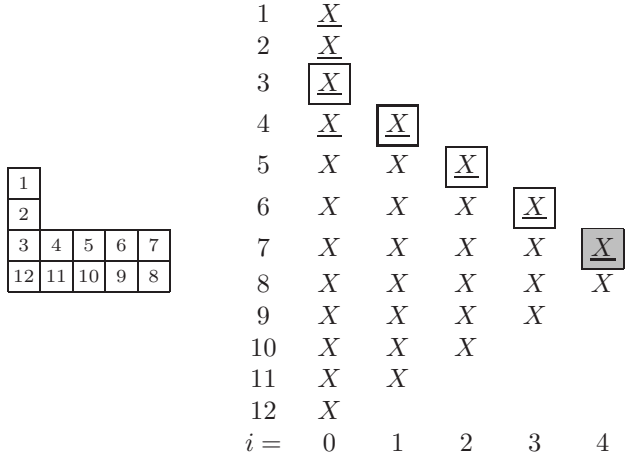

\newdimen\Squaresize \Squaresize=12pt \newdimen\Thickness
\Thickness=.3pt \newdimen\Correction \Correction=5pt
\begin{eqnarray*}
\begin{matrix} 
     \Young{ \Box{1} \cr
                     \Box{2} \cr 
                    \Box{3} &\Box{4} & \Box{5} &\Box{6} &  \Box{7}   \cr
                    \Box{12} &\Box{11} & \Box{10} &\Box{9} &  \Box{8} \cr
               \cr} 
     \end{matrix}
\hspace{1cm}
\begin{matrix} 
 1 & \underline{X} & & & &  &\\
 2 & \underline{X} & & & &  &\\    
 3 & \collbox{\underline{X}} & & & & & \\
 4 & \underline{X} & \collbox{\underline{X}} & & & &\\
 5 & X &X  & \collbox{\underline{X}}& & &\\
 6 & X & X &X  &\collbox{\underline{X}} & & \\
 7 & X & X & X & X & \collshadebox{\underline{X}} &\\
 8 & X & X & X  &X  & X &\\
 9 & X & X &  X  & X & &\\
 10 & X & X &X &  &  & \\
 11 & X & X & &  &   &  \\ 
 12 & X & & &   &  &  \\ 
i= & 0 & 1 & 2 & 3 & 4  
    \end{matrix}
\end{eqnarray*}
\caption{The regular filling and the Weyman diagram of $\lambda=(5,5,1,1)$.}
\label{figurequestion0}
\end{figure}

The following table, computed with Macaulay2, confirms our prediction that 
 the $275$ degree $7$ generators that should be in the generating set 
 according to the diagonal version of the conjecture, are not needed.
\begin{center}
 \begin{tabular}{c|l}
{\bf Degrees} &   {\bf Minimal number of generators}  \\
\hline& \\ 
1, 2, 3 & 1 \text{in each degree} \\
4 & 12  \\
5 & 54  \\
6 & 154  \\
7 & redundant \\
\end{tabular}
\end{center}
\end{example}

Theorems~\ref{t:reduction2}  and~\ref{t:reduction2bis} can be
reformulated in a suggestive geometrical way as special instances of
the following statement.
\begin{question} \label{lastquestion}
Let $\lambda$ be a partition and draw the Weyman diagram of
$\lambda$. If the $X's$ at the top of columns $1,2,\ldots, r$ are
collinear, and the line containing them passes through the point $(0,k)$,
then are the generators coming from columns $k+1,\ldots,r$ redundant ?

We have evidence that suggests that this statement is true: it was
proven to be true when $k=1$ in Theorem~\ref{t:reduction2}, for $k=2$
in Theorem~\ref{t:reduction2bis} Part 1, and for $k=3$ in Theorem
\ref{t:reduction2bis} Part 2, (see Figure~\ref{figurequestion0}: the collinear $X$'s have been surrounded).  For $k=4$, we used Macaulay2 to verify whether
the statement is still true for the smallest possible member of this
family, the partition (6,5,1,1,1) (see Figure~\ref{figurequestion}).  As predicted,
all degree $9$ generators are redundant.

\begin{figure}
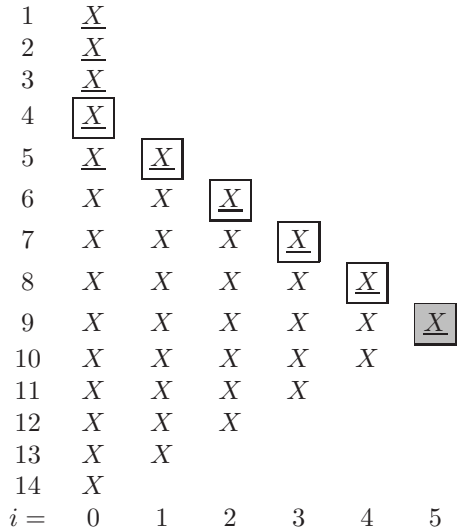

\begin{eqnarray*}
\begin{matrix}
 1 & \underline{X} & & & &  &\\
 2 & \underline{X} & & & &  &\\   
 3 & \underline{X} & & & & & \\
 4 & \collbox{\underline{X}} &  & & & &\\
 5 & \underline{X} & \collbox{\underline{X}} & & & &\\
 6 & X & X &\collbox{\underline{X}}  & & & \\
 7 & X & X & X&\collbox{\underline{X}} & &\\
 8 & X & X &X  &X  &\collbox{\underline{X}} &\\
 9 & X & X &X  &X  &X & \collshadebox{\underline{X}}\\
 10 & X & X & X &X  & X  &\\
 11 & X & X &X &X  &  & \\
 12 & X & X &  X & &   &  \\
 13 & X &X  & &  &  &  \\
 14 & X &  & &  &  & \\
i= & 0 & 1 & 2 & 3 & 4 & 5
\end{matrix}
    \end{eqnarray*}
    \caption{An evidence regarding the statement in Question
\ref{lastquestion} for $\lambda=(6,5,1,1,1)$}
\label{figurequestion}
\end{figure}

\begin{center}
    \begin{tabular}{l|l}
    {\bf Degrees} &{\bf Minimal number of generators}\\
    \hline& \\ 
    1, 2, 3, 4 & 1 in each degree \\
    5 & 14\\
    6 & 77\\
    7 & 273\\
    8 & 637\\
    9 & \text{redundant}
    \end{tabular}
\end{center}
\end{question}

\idiot{ An example of the Weyman diagram of a partition in the first family.
Here $g=3$ $\lambda=(5,5,5,1)$ and $n=14$
\begin{eqnarray*}
\begin{matrix} 
 1 & X & & & &  &\\
 2 & \collbox{X} & & & &  &\\    
 3 & X & & & & & \\
 4 & X & \collbox{X} & & & &\\
 5 & X & X & & & &\\
 6 & X & X & \collbox{X} & & & \\
 7 & X & X & X &  & &\\
 8 & X & X & X & \collbox{X} & &\\
 9 & X & X & X & X & &\\
 10 & X & X &X &X & \collbox{X} & \\
 11 & X & X &X &X & X &  \\ 
 12 & X & X &X & X & X &  \\ 
 13 & X & X  & X & X & &\\
 14 & X & X&X &  & &\\  
 15 & X & X& &  & &\\
 16 & X & & &  & &\\
i= & 0 & 1 & 2 & 3 & 4 & 5 
    \end{matrix}
\end{eqnarray*}
Note that again we have a phenomena of collinearity BUT 
the line does not pass through $(0,1)$ but by $(0,2)$. 
Note that in this case, we need generators coming from columns $0,1,$ and $2$.
 }

\idiot{ An example of the Weyman diagram of a partition in the second family.
Here $g=5$ $\lambda=(4,4,4,4,1,1)$ and $n=18$
\begin{eqnarray*}
\begin{matrix} 
 1 & X & & & &  &\\
 2 & X & & & &  &\\    
 3 &  \collbox{X}  & & & & & \\
 4 & X &  & & & &\\
 5 & X &  & & & &\\
 6 & X &  \collbox{X}  &  & & & \\
 7 & X & X & &  & &\\
 8 & X & X &  &  & &\\
 9 & X & X &  \collbox{X}  &  & &\\
 10 & X & X &X &  &  & \\
 11 & X & X &X &  &   &  \\ 
 12 & X & X &X &  \collbox{X}  &  &  \\ 
 13 & X & X  & X & X & &\\
 14 & X & X&X &  X& &\\
 15 & X & X& X & X  &  &\\
 16 & X & X & X &  & &\\
 17 & X & X& &  & &\\
 18 & X & & &  & &\\
    \end{matrix}
\end{eqnarray*}
Note that again we have a phenomena of collinearity BUT 
the line does not pass through $(0,1)$ nor $(0,2)$ but by $(0,3)$. 
Note that in this case, we need generators coming from columns $0,1,2$ and $3$.
 }

\subsection{Weyman's original conjecture}

To finish our work, we focus our attention at the original conjecture
of Weyman.  It seems plausible that those partitions that give
counterexamples to the diagonal version of Weyman's conjecture are
also counterexamples to Weyman's original conjecture. We used
Macaulay2 to verify if this was the case for the smallest shape
in the families described in Counterexample~\ref{counterdiagonal}.

\begin{counterexample}[Counterexample to Weyman's original conjecture] 
Consider the partition $(4,3,1)$ whose Weyman diagram is represented in
Figure~\ref{figurecounter}.  The points $(1,3), (2,4)$ and $(3,5)$ are
collinear, but the line that contains them does not pass through
$(0,1)$. So according to Weyman's conjecture, all these cells 
contribute generators to a minimal generating set of
$\mathcal{J}_{(4,3,1)}$.  However, Theorem~\ref{t:reduction2bis}
suggests that the generators coming from cell $(3,5)$ may be
redundant.
\begin{figure}
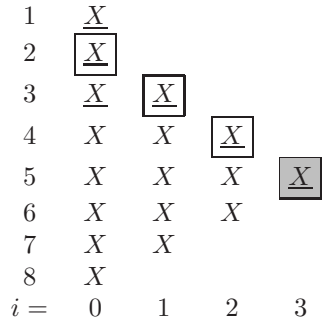

\begin{eqnarray*}
\begin{matrix} 
 1 & \underline{X} & & & &  &\\
 2 & \collbox{\underline{X}} & & & &  &\\    
 3 & \underline{X} & \collbox{\underline{X}}& & & & \\
 4 & X &X  &\collbox{\underline{X}}& & &\\
 5 & X & X &X & \collshadebox{\underline{X}}& &\\
 6 & X & X &X  & & & \\
 7 & X & X & && &\\
 8 & X &  & && &\\
i= & 0 & 1 & 2 & 3
    \end{matrix}
\end{eqnarray*}
\caption{ A counterexample to Weyman's original conjecture:  $ (4,3,1)$. \label{figurecounter} }
\end{figure}
%

Using Macaulay 2, we computed the minimal generating set for
$\mathcal{J}_{(4,3,1)}$ and verified that this is indeed the case. We
conclude that $(4,3,1)$ is a counterexample to Weyman's original
conjecture.

\begin{center}
 \begin{tabular}{c|l|lll}
{\bf Degrees} &  {\bf Weyman's conjecture} & {\bf Minimal number of generators}  \\
\hline& \\ 
1 & 1  &1\\
 2 & 1 &1  \\ 
 3 & 64 & 64  \\
4 & 720 & 720 \\
 5 & 2352 & redundant \\
\hline&&\\
{\bf Total}& 3138 & 786 
\end{tabular}
 \end{center}

To summarize, in this particular case, Weyman's conjecture predicts
that we need 3138 generators, but only $786$ of them are really
necessary.

Unfortunately, even large servers were not able to handle slightly
larger examples, so at this point we do not know if other partitions
in the families described earlier are counterexamples to Weyman's
original conjecture. 
\end{counterexample}

We end the paper with a natural question.
\begin{question}
Does the statement of Question~\ref{lastquestion} hold for $\mathcal{J}_\lambda$ ?

\end{question}

\section*{Acknowledgments}

We wish to thank Jerzy Weyman for many interesting conversations and
suggestions, as well as his interest in our project.  At the same time
that we were working on this project, he showed independently that
$(4,3,1)$ is indeed a counterexample to his Conjecture 5.1 using
geometric reasoning \cite{W3}. Moreover, he has shown that Conjecture
5.1 holds for all the other partitions of $n \le 9$ except for the
three partitions of $9$ belonging to the shapes described in
Counterexample~\ref{counterdiagonal}.  We also wish to Mark Shimozono
pointing out some references and Emmanuel Briand for his help during
this project.

\idiot{It will be ideal to use M2 to see what happens for partitions
$(4,4,1), (5,3,1), (4,3,1,1)$

--Crashed! Can't handle 9 variables on 24G server!}


\end{document}